\documentclass[oneside,english]{siamart190516}
\usepackage[T1]{fontenc}
\usepackage[latin9]{inputenc}
\usepackage{array}
\usepackage{booktabs}
\usepackage{units}
\usepackage{multirow}
\usepackage{amsmath}
\usepackage{amssymb}
\usepackage{graphicx}

\makeatletter

\providecommand{\tabularnewline}{\\}

\usepackage{times}
\usepackage{lipsum}
\usepackage{amsfonts}
\usepackage{graphicx}
\usepackage{epstopdf}
\usepackage{algorithmic}
\usepackage[symbol]{footmisc}
\ifpdf%
  \DeclareGraphicsExtensions{.eps,.pdf,.png,.jpg}
\else
  \DeclareGraphicsExtensions{.eps}
\fi

\usepackage{booktabs}

\newsiamremark{remark}{Remark}

\@ifundefined{showcaptionsetup}{}{%
 \PassOptionsToPackage{caption=false}{subfig}}
\usepackage{subfig}
\makeatother

\usepackage{babel}

\begin{document}
\title{Optimal and Low-Memory Near-Optimal Preconditioning of \\
Fully Implicit Runge-Kutta Schemes for Parabolic PDEs}
\author{Xiangmin Jiao\footnotemark[1]\ \footnotemark[2] \and Xuebin Wang\footnotemark[1]
\and Qiao Chen\footnotemark[1] }

\maketitle
\footnotetext[1]{Department of Applied Mathematics \& Statistics and Institute for Advanced Computational Science, Stony Brook University, Stony Brook, NY 11794, USA.}
\footnotetext[2]{Corresponding author. Email: \email{xiangmin.jiao@stonybrook.edu}.}
\headers{Optimal and Near-Optimal Preconditioning of FIRK}{X. Jiao, X. Wang, and Q. Chen}
\begin{abstract}
Runge-Kutta (RK) schemes, especially Gauss-Legendre and some other
fully implicit RK (FIRK) schemes, are desirable for the time integration
of parabolic partial differential equations due to their A-stability
and high-order accuracy. However, it is significantly more challenging
to construct optimal preconditioners for them compared to diagonally
implicit RK (or DIRK) schemes. To address this challenge, we first
introduce mathematically optimal preconditioners called \emph{block
complex Schur decomposition} (\emph{BCSD}), \emph{block real Schur
decomposition} (\emph{BRSD}), and \emph{block Jordan form} (\emph{BJF}),
motivated by block-circulant preconditioners and Jordan form solution
techniques for IRK. We then derive an efficient, near-optimal \emph{singly-diagonal
approximate BRSD }(\emph{SABRSD}) by approximating the quasi-triangular
matrix in real Schur decomposition using an optimized upper-triangular
matrix with a single diagonal value. A desirable feature of SABRSD
is that it has comparable memory requirements and factorization (or
setup) cost as singly DIRK (SDIRK). We approximate the diagonal blocks
in these preconditioning techniques using an incomplete factorization
with (near) linear complexity, such as multilevel ILU, ILU(0), or
a multigrid method with an ILU-based smoother. We apply the block
preconditioners in right-preconditioned GMRES to solve the advection-diffusion
equation in 3D using finite element and finite difference methods.
We show that BCSD, BRSD, and BJF significantly outperform other preconditioners
in terms of GMRES iterations, and SABRSD is competitive with them
and the prior state of the art in terms of computational cost while
requiring the least amount of memory.
\end{abstract}

\section{Introduction\label{sec:Introduction}}

We consider the solution of time-dependent parabolic partial differential
equations (PDEs), such as the advection-diffusion (AD) equation for
$u:\Omega\times[0,T]\rightarrow\mathbb{R}$,
\begin{equation}
u_{t}-\boldsymbol{\nabla}\cdot(\mu\boldsymbol{\nabla}u)+\boldsymbol{v}\cdot\boldsymbol{\nabla}u=f,\label{eq:advection-diffusion}
\end{equation}
where $\mu\geq0$ denotes the diffusion coefficient, $\boldsymbol{v}$
denotes a velocity field, $f$ denotes some source term, and $u_{t}$
denotes the temporal derivative. In the extreme case of $\boldsymbol{v}=\boldsymbol{0}$,
\eqref{eq:advection-diffusion} reduces to the heat equation. Typically,
the spatial discretization uses finite difference methods (FDM) or
finite element methods (FEM). We assume the spatial discretization
is well-posed, and \eqref{eq:advection-diffusion} is diffusion dominant
in the sense that the cell P\'eclet number $\text{Pe}_{h}=2h\Vert\boldsymbol{v}\Vert/\text{\ensuremath{\mu}}\lesssim1$
\cite{ern2013theory}, where $h$ denotes a characteristic edge length
of the mesh.

These methods convert \eqref{eq:advection-diffusion} into a system
of stiff ordinary differential equations (ODEs), 
\begin{equation}
\boldsymbol{M}\boldsymbol{u}_{t}(t)=-\boldsymbol{K}\boldsymbol{u}(t)+\boldsymbol{f}(t),\label{eq:stiff-ODEs}
\end{equation}
where $\boldsymbol{M}\in\mathbb{R}^{m\times m}$ and $\boldsymbol{K}\in\mathbb{R}^{m\times m}$
denote the mass and stiffness matrices (Section~\ref{subsec:Mass-and-stiffness}),
correspondingly, and $\boldsymbol{f}:[0,T]\rightarrow\mathbb{R}^{m}$.
For long-time integration, fourth and higher-order accurate FDM and
FEM (including spectral elements \cite{karniadakis2005spectral})
are often used. Therefore, it is desirable to solve \eqref{eq:stiff-ODEs}
using high-order implicit time-integration schemes that are A-stable,
so that larger time steps can be used without compromising stability
and accuracy.

The Gauss-Legendre, Radau IIA, Lobatto IIIC, and some other implicit
Runge-Kutta (IRK) are attractive for their A-stability and high-order
accuracy; see, e.g., \cite{butcher2016numerical,hairer1996solving,iserles2009first}.
An $s$-stage Runge-Kutta (RK) scheme can be expressed by the Butcher
tableau $\begin{array}{ccc}
\boldsymbol{c} & \vline & \boldsymbol{A}\\
\hline  & \vline & \boldsymbol{b}^{T}
\end{array}$ \cite{butcher2016numerical}, where $\boldsymbol{A}\in\mathbb{R}^{s\times s}$,
$\boldsymbol{c}\in\mathbb{R}^{s}$, and $\boldsymbol{b}\in\mathbb{R}^{s}$.
For A-stable IRK schemes with an invertible $\boldsymbol{A}$, the
real part of the eigenvalues of $\boldsymbol{A}$ are all positive
\cite[p. 402]{hairer1996solving}. Given an $s$-stage IRK and time
step $\delta t$, \eqref{eq:stiff-ODEs} leads to an $sm\times sm$
linear system $\boldsymbol{\mathcal{A}}\mathcal{U}=\mathcal{B}$ with
\begin{equation}
\boldsymbol{\mathcal{A}}=\boldsymbol{I}_{s}\otimes\boldsymbol{M}+\delta t\boldsymbol{A}\otimes\boldsymbol{K},\label{eq:coeff_matrix}
\end{equation}
where $\boldsymbol{I}_{s}$ is the $s\times s$ identity matrix, $\otimes$
denotes the Kronecker-product operator, and $\mathcal{B}$ depends
on the solution in the previous time step, the boundary conditions,
and the source term $\boldsymbol{f}$; see, e.g., \cite[p. 223]{donea2003finite}.
Unless otherwise noted, we use boldface and regular calligraphic fonts
to indicate block matrices and vectors, respectively. Let $b_{i}$
denote the entries in $\boldsymbol{b}$ and $\boldsymbol{k}_{i}\in\mathbb{R}^{m}$
denote the subvectors in $\mathcal{U}$ corresponding to stage $i$
for $i=1,2,\dots,s$. The solution at time step $k+1$ is then $\boldsymbol{u}_{k+1}=\boldsymbol{u}_{k}+\delta t\sum_{i=1}^{s}b_{i}\boldsymbol{k}_{i}$.
We aim to construct a \emph{right preconditioner }$\boldsymbol{\mathcal{M}}\in\mathbb{R}^{sm\times sm}$
and to solve the preconditioned system
\begin{equation}
\boldsymbol{\mathcal{A}}\boldsymbol{\mathcal{M}}^{-1}\mathcal{V}=\mathcal{B}\label{eq:right-preconditioner}
\end{equation}
using a Krylov-subspace (KSP) method (such as GMRES \cite{saad1986gmres}),
and then $\mathcal{U}=\boldsymbol{\mathcal{M}}^{-1}\mathcal{V}$.

Among the IRK schemes, the fully implicit RK (FIRK) schemes have the
highest-order accuracy. In particular, the $s$-stage Gauss-Legendre
(GL) \cite{iserles2009first} (aka Gauss \cite{hairer1996solving})
are of order $2s$. However, the Butcher matrix $\boldsymbol{A}$
for FIRK is typically full, so it is more challenging to construct
effective and robust preconditioners for them compared to diagonally
implicit Runge-Kutta (or DIRK, sometimes also referred to as semi-implicit
RK \cite{butcher2016numerical}) schemes \cite{najafi2013low}, for
which $\boldsymbol{A}$ and $\boldsymbol{\mathcal{A}}$ are lower
triangular and block lower triangular, respectively. Among the DIRK
schemes, singly DIRK (SDIRK) \cite[IV.6]{hairer1996solving} are often
used, because the diagonal entries in $\boldsymbol{A}$ are all equal
and in turn they lead to lower factorization (or setup) cost and memory
requirement. However, an $s$-stage SDIRK is limited to order-$(s+1)$,
which is far lower than an $s$-stage GL for large $s$. Hence, there
are significant interests in developing effective preconditioners
for FIRK; see, e.g., \cite{chen2014splitting,jameson2017evaluation,pazner2017stage,staff2006Spreconditioning,van1997triangularly}.
Those preconditioners typically have a similar cost as DIRK or SDIRK
per iteration, but the number of KSP iterations may be far from optimal.

In this work, we introduce optimal and limited-memory near-optimal
preconditioning techniques for FIRK. Our approach is novel in two
main aspects. First, we introduce three block preconditioners called
BCSD, BRSD, and BJF, based on the complex and real Schur decompositions
and the Jordan decomposition, respectively. These preconditioners
are motivated by block circulant preconditioners \cite{chan2001strang}
and the Jordan form solution techniques for IRK \cite{bickart1977efficient,butcher1976implementation},
and they are mathematically optimal in exact arithmetic. Second, we
introduce a \emph{singly-diagonal approximate BRSD} (SABRSD) preconditioner,
which reduces the memory requirement and the factorization times compared
to the aforementioned optimal preconditioners. We approximate the
diagonal blocks in these block preconditioners using near-linear-complexity
approximate factorizations, such as the multilevel-ILU-based HILUCSI
\cite{chen2021hilucsi}, the commonly used ILU(0) \cite{kanevsky2007application,pazner2017stage,persson2009scalable},
or a multigrid preconditioner with an ILU-based smoother \cite{hypre2021,hysom2001scalable}.
Our experimental results show that BCSD, BRSD, and BJF significantly
outperform other preconditioners in terms of the number of GMRES iterations.
SABRSD compares favorably to the prior state of the art in terms of
computational cost while requiring the least amount of memory. This
work focuses on linear PDEs with time-invariant coefficients and source
terms; however, it may be extended to nonlinear PDEs with certain
linearization within each time step.

The remainder of the paper is organized as follows. In Section~\ref{sec:Related-works},
we review some background on FDM and FEM, optimal preconditioning,
and some existing preconditioners for FIRK. In Section~\ref{sec:Optimal-preconditioners},
we introduce the mathematically optimal BCSD, BRSD, and BJF. In Section~\ref{sec:Near-Optimal-Preconditioners},
we introduce the near-optimal SABRSD preconditioners based on BRSD
by deriving a novel optimization strategy. In Section~\ref{sec:Numerical-Results},
we present numerical results with the proposed block preconditioners,
verify their (near) optimality, and compare them with the prior state-of-the-art
preconditioners. Finally, Section~\ref{sec:conclusions} concludes
the paper with a brief discussion of future research directions.

\section{Preliminaries and Related Works\label{sec:Related-works}}

We start by reviewing some basics of preconditioning techniques, especially
in the context of FIRK.

\subsection{Optimal preconditioning}

Our aim in this work is to develop optimal and near-optimal right
preconditioners that are stable, accurate, and efficient for FIRK.
We prefer right preconditioning in that it does not affect the norm
of the residual vector, which is preferable for GMRES \cite{ghai2019comparison}.
To measure the stability and accuracy of preconditioners, we use a
definition based on that in \cite{jiao2020approximate}, which generalized
some other measures reviewed in \cite{benzi2002preconditioning}.
\begin{definition}
\label{def:epsilon-accuracy} Given $\boldsymbol{\mathcal{A}}\in\mathbb{R}^{n\times n}$,
$\boldsymbol{\mathcal{G}}$ is an $\epsilon$\emph{-accurate} \emph{right-preconditioning
operator} (\emph{RPO}) for $\boldsymbol{\mathcal{A}}$ if there exists
a nonsingular $\boldsymbol{\mathcal{X}}\in\mathbb{R}^{n\times n}$
(or $\in\mathbb{C}^{n\times n}$) such that
\begin{equation}
\left\Vert \boldsymbol{\mathcal{X}}^{-1}\left(\boldsymbol{\mathcal{A}}\boldsymbol{\mathcal{G}}\right)\boldsymbol{\mathcal{X}}-\begin{bmatrix}\boldsymbol{I}_{r}\\
 & \boldsymbol{0}
\end{bmatrix}\right\Vert =\epsilon<1,\label{eq:norm-bound}
\end{equation}
where $r=\text{rank}(\boldsymbol{\mathcal{A}})$. A class of RPO is
$\epsilon$\emph{-accurate} if $\epsilon$ approaches 0 as its control
parameters are tightened. $\boldsymbol{\mathcal{G}}$ is \emph{stable
}if \emph{$\kappa(\boldsymbol{\mathcal{X}})\leq C$ }for some bounded\emph{
$C>1$}. An $\epsilon$-accurate RPO is \emph{mathematically optimal}
if $\epsilon=0$.
\end{definition}

Definition~\ref{def:epsilon-accuracy} is general in that it applies
to both singular and nonsingular systems. For nonsingular matrices,
which are the focus of this work, the reader can simply interpret
$\boldsymbol{\mathcal{G}}=\boldsymbol{\mathcal{M}}^{-1}$ for a nonsingular
$\boldsymbol{\mathcal{M}}\in\mathbb{R}^{n\times n}$ (or $\in\mathbb{C}^{n\times n}$),
where $\boldsymbol{\mathcal{M}}$ is the \emph{right preconditioner},
and $\boldsymbol{\mathcal{G}}$ is mathematically optimal if and only
if $\boldsymbol{\mathcal{A}}\boldsymbol{\mathcal{M}}^{-1}=\boldsymbol{\mathcal{I}}$.
More generally, however, $\boldsymbol{\mathcal{G}}$ does not need
to be the inverse of a nonsingular matrix \cite{jiao2020approximate};
in this case, an $\epsilon$-accurate RPO guarantees the breakdown-free
of GMRES for consistent systems \cite[Theorem 3.9]{jiao2020approximate}.
More intuitively, since $\left\Vert \boldsymbol{\mathcal{X}}^{-1}\left(\boldsymbol{\mathcal{A}}\boldsymbol{\mathcal{M}}^{-1}\right)\boldsymbol{\mathcal{X}}-\boldsymbol{I}\right\Vert \leq\kappa(\boldsymbol{\mathcal{X}})\left\Vert \boldsymbol{\mathcal{M}}^{-1}\right\Vert \left\Vert \boldsymbol{\mathcal{A}}-\boldsymbol{\mathcal{M}}\right\Vert $,
$\epsilon$-accuracy for nonsingular matrices depends on how close
$\boldsymbol{\mathcal{M}}$ is to $\boldsymbol{\mathcal{A}}$ and
how well conditioned $\boldsymbol{\mathcal{X}}$ and $\boldsymbol{\mathcal{M}}$
are in \eqref{eq:norm-bound}. For this reason, Definition~\ref{def:epsilon-accuracy}
is useful in assessing preconditioners qualitatively. In addition,
we will also use Definition~\ref{def:epsilon-accuracy} to derive
objective functions in optimizing $\boldsymbol{\mathcal{M}}$ quantitatively.

Efficiency is multifaceted, including the factorization of $\boldsymbol{\mathcal{M}}$,
the solve $\boldsymbol{\mathcal{M}}^{-1}\mathcal{\mathcal{V}},$ the
number of KSP iterations, the sparse matrix-vector (SpMV) multiplication
$\boldsymbol{\mathcal{A}}\mathcal{U}$, and the memory requirement.
From a practical point of view, we use the following criterion, assuming
that the solve time of $\boldsymbol{\mathcal{M}}^{-1}\mathcal{\mathcal{V}}$
dominates that of SpMV $\boldsymbol{\mathcal{A}}\mathcal{\mathcal{U}}$.
\begin{definition}
\label{def:near-optimality}A preconditioner $\boldsymbol{\mathcal{M}}$
is \emph{near-optimal }if it is $\epsilon$-accurate and it has (near)
linear complexity to factorize (or setup) and solve (per KSP iteration).
\end{definition}

By linear complexity, we refer to complexities in both time and space.
In the context of IRK, the preconditioner $\boldsymbol{\mathcal{M}}$
can be reused if the mesh and $\delta t$ remain the same. Hence,
the factorization time can be amortized over many time steps. However,
the memory requirement and the solve time for $\boldsymbol{\mathcal{M}}^{-1}\mathcal{V}$
must be (near) linear. Of course, (near) linear time is impractical
in Definition~\ref{def:near-optimality} for $\epsilon=0$, so our
goal is to make $\epsilon$ as small as possible.

\subsection{\label{subsec:Mass-and-stiffness}Mass and stiffness matrices from
FEM and FDM}

Let us first briefly review the finite element methods (FEM). Given
$\Omega\in\mathbb{R}^{d}$ with a piecewise smooth boundary, a set
of trial (basis) functions $\{\phi_{j}(\boldsymbol{x}):\Omega\rightarrow\mathbb{R}\mid1\leq j\leq m\}$
and a set of test functions $\{\psi_{i}(\boldsymbol{x}):\Omega\rightarrow\mathbb{R}\mid1\leq i\leq m\}$,
an FEM (or a \emph{weighted-residual method}) approximates the continuum
solution $u(\boldsymbol{x},t)$ by $\sum_{j=1}^{m}U_{j}\left(t\right)\phi_{j}(\boldsymbol{x})$
and requires the residual to be orthogonal to the test functions $\{\psi_{i}\}$.
After applying integration by parts, \eqref{eq:advection-diffusion}
is converted into a system of ODEs \eqref{eq:stiff-ODEs}, where 
\begin{equation}
\boldsymbol{M}_{\text{FEM}}=\left[\int_{\Omega}\phi_{j}\psi_{i}\,\text{d}\boldsymbol{x}\right]_{ij}\text{ and}\quad\boldsymbol{K}_{\text{FEM}}=\left[\int_{\Omega}\mu\boldsymbol{\nabla}\phi_{j}\cdot\boldsymbol{\nabla}\psi_{i}+\left(\boldsymbol{v}\cdot\boldsymbol{\nabla}\phi_{j}\right)\psi_{i}\,\text{d}\boldsymbol{x}\right]_{ij}.\label{eq:mass-striff-FEM}
\end{equation}
In the Galerkin FEM, $\{\phi_{j}\}=\{\psi_{i}\}$. $\boldsymbol{M}_{\text{FEM}}$
is called the \emph{mass matrix}, and $\boldsymbol{K}_{\text{FEM}}$
is called the \emph{stiffness matrix} when $\boldsymbol{v}=\boldsymbol{0}$. 

An FDM is a collocation-type method, which we can interpret as a \emph{generalized
weighted residual method }\cite{conley2020hybrid}, where the ``test
functions'' are the Dirac delta functions at the nodes. The finite-difference
formula at the $i$th node can be constructed from a corresponding
set of local basis functions $\{\phi_{ij}(\boldsymbol{x}):\Omega\rightarrow\mathbb{R}\mid1\leq i,j\leq m\}$.
We then have 
\begin{equation}
\boldsymbol{M}_{\text{FDM}}=\boldsymbol{I}\quad\text{and}\quad\boldsymbol{K}_{\text{FDM}}=\left[-\boldsymbol{\nabla}\cdot(\mu\boldsymbol{\nabla}\phi_{ij})+\boldsymbol{v}\cdot\boldsymbol{\nabla}\phi_{ij}\right]_{ij}.\label{eq:mass-stiff-FDM}
\end{equation}
For convenience, we also refer to $\boldsymbol{M}_{\text{FDM}}$ as
the mass matrix in FDM, and we informally refer to both $\boldsymbol{K}_{\text{FEM}}$
and $\boldsymbol{K}_{\text{FDM}}$ also as the ``stiffness'' matrix
even when $\boldsymbol{v}\neq\boldsymbol{0}$. For simplicity, we
will omit their subscript when there is no confusion. Note that $-\boldsymbol{K}$
is sometimes referred to as the \emph{Jacobian matrix} in \eqref{eq:stiff-ODEs}
and is often denoted as $\boldsymbol{J}$ in the ODE literature. Assuming
$\mu>0$ and well-posedness of the spatial discretization, $\boldsymbol{K}$
for FEM is positive definite in that $\boldsymbol{K}+\boldsymbol{K}^{T}$
is symmetric and positive definite (SPD). For Galerkin FEM, $\boldsymbol{M}$
is SPD with a Cholesky factorization $\boldsymbol{M}=\boldsymbol{R}^{T}\boldsymbol{R}$.
Hence, $\boldsymbol{M}^{-1}\boldsymbol{K}$ is also positive definite
for its similarity with $\boldsymbol{R}^{-T}\boldsymbol{K}\boldsymbol{R}^{-1}$,
which has the same inertia as $\boldsymbol{K}$ \cite[p. 448]{Golub13MC}.
For FDM, $\boldsymbol{M}^{-1}\boldsymbol{K}=\boldsymbol{K}$, which
is also positive definite on sufficiently fine meshes \cite{jovanovic2013analysis}. 

The preceding spatial discretizations lead to a system of ODEs, which
is stiff because $\kappa(\boldsymbol{K})=\mathcal{O}(h^{-2})$ and
$\kappa(\boldsymbol{M})=\mathcal{O}(1)$ for both FEM on quasiuniform
mesh \cite{ern2013theory} and FDM on uniform mesh \cite{leveque2007finite},
where $h$ denotes some characteristic edge length of the mesh. Due
to the positive definiteness of $\boldsymbol{M}^{-1}\boldsymbol{K}$,
an A-stable IRK ensures an accurate and stable solution of \eqref{eq:stiff-ODEs}
with sufficiently small time steps and grid resolution. An IRK scheme
leads to a linear system with the coefficient matrix \eqref{eq:coeff_matrix}.
Let $a_{ij}$ denote the entries in the Butcher matrix $\boldsymbol{A}$.
If $\boldsymbol{A}$ has positive diagonal entries $a_{ii}$, each
diagonal block of $\boldsymbol{\mathcal{A}}$ has the form $\boldsymbol{M}+a_{ii}\delta t\boldsymbol{K}$,
which is also positive definite.

When preconditioning FIRK, the fundamental question is how to express
an approximate inverse of $\boldsymbol{\mathcal{A}}$ in terms of
the approximate factorization of linear combinations of $\boldsymbol{M}$
and $\boldsymbol{K}$ as in \eqref{eq:coeff_matrix}. The following
relationship between the stiffness and mass matrices will turn out
to be useful in analyzing existing preconditioners.
\begin{lemma}
\label{lem:mass-stiff-relationship}For diffusion-dominated AD equations,
$\Vert\boldsymbol{K}\Vert=\mathcal{O}(h^{-2})\Vert\boldsymbol{M}\Vert$
for well-posed FEM and FDM.
\end{lemma}

We omit the proof; see, e.g., \cite{ern2013theory} for FEM and \cite{conley2020hybrid}
for a unified analysis of FEM and FDM. Lemma~\ref{lem:mass-stiff-relationship}
leads to the following proposition.
\begin{proposition}
\label{prop:asymptotic-optimal}For diffusion-dominated AD equations,
$\delta t\boldsymbol{A}\otimes\boldsymbol{K}$ and $\boldsymbol{I}_{s}\otimes\boldsymbol{M}$
converge to mathematically optimal preconditioners as $h^{-2}\delta t$
approches $\infty$ and $0$, respectively.
\end{proposition}

\begin{proof}
As $h^{-2}\delta t$ approaches $\infty$ and $0$, $\boldsymbol{M}$
and $\boldsymbol{K}$ in \eqref{eq:coeff_matrix} can be omitted,
respectively.
\end{proof}
In general, $\delta t\gtrsim\mathcal{O}(h)$ when using implicit schemes,
and hence the limiting case of $h^{-2}\delta t\rightarrow0$ has little
practical value. For very large $h^{-2}\delta t$, Proposition~\ref{prop:asymptotic-optimal}
suggests a \emph{physics-based} (or more precisely, \emph{PDE-based})
\emph{nearest Kronecker product} (PNKP) preconditioner,
\begin{equation}
\boldsymbol{\mathcal{M}}_{\text{PNKP}}=\delta t\boldsymbol{A}\otimes\boldsymbol{K}.\label{eq:PNKP}
\end{equation}
It is convenient to approximate the inverse of PNKP due to the inverse
rule of the Kronecker product \cite[Section 1.3.6]{Golub13MC}, 
\begin{equation}
\left(\boldsymbol{S}\otimes\boldsymbol{T}\right)^{-1}=\boldsymbol{S}^{-1}\otimes\boldsymbol{T}^{-1}.\label{eq:inverse-of-Kronecker-product}
\end{equation}
Hence, in PNKP, one only needs to (approximately) factorize $\boldsymbol{K}$.
However, $\boldsymbol{\mathcal{M}}_{\text{PNKP}}$ is expected to
be effective only as $h^{-2}\delta t$ approaches $\infty$, for example,
when $\delta t\gg\sqrt{h}$, so it is not optimal for general purpose.
In the work, we will focus on moderately large time steps, i.e., $\delta t=\mathcal{O}(h^{\alpha})$
for some $0<\alpha\leq1$. We will use $\boldsymbol{\mathcal{M}}_{\text{PNKP}}$
as a baseline in comparison with other preconditioners in terms of
efficiency and effectiveness for large $h^{-2}\delta t$.

\subsection{State-of-the-Art Preconditioners for FIRK\label{subsec:State-of-the-Art-Preconditioners}}

We briefly review some existing preconditioners of FIRK, which were
representative of state of the art prior to this work. 

\subsubsection{Block diagonal and block triangular preconditioners}

One of the simplest and the most effective preconditioners for FIRK
is the \emph{block Gauss-Seidel} (\emph{BGS}), i.e.,
\begin{equation}
\boldsymbol{\mathcal{M}}_{\text{BGS}}=\boldsymbol{I}_{s}\otimes\boldsymbol{M}+\delta t\boldsymbol{L}\otimes\boldsymbol{K},\label{eq:BGS}
\end{equation}
where $\boldsymbol{L}$ is the lower triangular part of $\boldsymbol{A}$.
From Proposition~\ref{prop:asymptotic-optimal}, it is easy to see
that BGS is mathematically optimal as $h^{-2}\delta t\rightarrow0$,
but not so as $h^{-2}\delta t\rightarrow\infty$. Hence, we expect
BGS to perform well for small time steps but worse for large time
steps. BGS was investigated by van der Houwen and de Swart in \cite{van1997triangularly}
along with other block-triangular preconditioners. They took advantage
of the fact that the Butcher matrices from FIRK, such as those of
Gauss-Legendre, Radau IIA, and Lobatto IIIB, often have a dominant
lower triangular part. Therefore, BGS performs significantly better
than block diagonal \cite{mardal2007order} (BD, aka block Jacobi
\cite{persson2009scalable}) or block upper-triangular preconditioners.
There were some attempts on improving BGS, such as \cite{staff2006Spreconditioning},
by replacing $\boldsymbol{L}$ with a different lower-triangular matrix
$\tilde{\boldsymbol{L}}$ that minimizes $\kappa(\boldsymbol{A}\tilde{\boldsymbol{L}}^{-1})$
under the constraint that $\text{diag}(\tilde{\boldsymbol{L}})=\text{diag}(\boldsymbol{L})$.
Our numerical results, however, show that the preconditioners in \cite{staff2006Spreconditioning}
often under-performed BGS. Hence, although simple, BGS is representative
of the state-of-the-art block triangular preconditioners. $\boldsymbol{\mathcal{M}}_{\text{BGS}}$
has a similar cost per KSP iteration compared to PNKP. However, for
the $s$-stage GL schemes, BGS needs to (approximately) factorize
$\lceil\nicefrac{s}{2}\rceil$ distinct diagonal blocks because the
diagonal of $\boldsymbol{A}$ has $\lceil\nicefrac{s}{2}\rceil$ distinct
values.

\subsubsection{Block circulant preconditioners}

Motivated by a Strang-type circulant preconditioner for Toeplitz-like
matrices \cite{bertaccini2000circulant}, Chan et al. \cite{chan2001strang}
proposed a Strang-type block-circulant (BC) preconditioner for linear
multistep formulas (LMF). Although high-order LMF is less desirable
than implicit RK schemes due to their lack of A-stability \cite{dahlquist1963special},
the idea of BC preconditioner can be adapted to FIRK \cite{chen2015generalized,zhang2011strang}.
In this context, the preconditioner has the form
\begin{align}
\boldsymbol{\mathcal{M}}_{\text{BC}} & =\boldsymbol{I}_{s}\otimes\boldsymbol{M}+\delta t\boldsymbol{C}\otimes\boldsymbol{K}\label{eq:BC-1}\\
 & =(\boldsymbol{F}\otimes\boldsymbol{I}_{n})(\boldsymbol{I}_{s}\otimes\boldsymbol{M}+\delta t\boldsymbol{\Lambda}\otimes\boldsymbol{K})(\boldsymbol{F}^{H}\otimes\boldsymbol{I}_{n})\label{eq:BC-2}
\end{align}
where $\boldsymbol{F}\in\mathbb{C}^{s\times s}$ is the normalized
Fourier matrix (i.e., $\boldsymbol{F}\boldsymbol{F}^{H}=\boldsymbol{I}$),
and $\boldsymbol{C}=\boldsymbol{F}\boldsymbol{\Lambda}\boldsymbol{F}^{H}$
is the Schur decomposition of a circulant matrix that approximates
$\boldsymbol{A}$. From the inverse rule \eqref{eq:inverse-of-Kronecker-product}
and the product rule of the Kronecker product \cite[Section 1.3.6]{Golub13MC},
\begin{equation}
(\boldsymbol{S}\otimes\boldsymbol{X})(\boldsymbol{T}\otimes\boldsymbol{Y})=(\boldsymbol{S}\boldsymbol{T})\otimes(\boldsymbol{X}\boldsymbol{Y}),\label{eq:product-rule}
\end{equation}
we then obtain a closed-form expression
\[
\boldsymbol{\mathcal{M}}_{\text{BC}}^{-1}=(\boldsymbol{F}\otimes\boldsymbol{I}_{n})(\boldsymbol{I}_{s}\otimes\boldsymbol{M}+\delta t\boldsymbol{\Lambda}\otimes\boldsymbol{K})^{-1}(\boldsymbol{F}^{H}\otimes\boldsymbol{I}_{n}),
\]
where the matrix in the middle is block diagonal, of which the diagonal
blocks need to be (approximately) factorized. The effectiveness of
$\boldsymbol{\mathcal{M}}_{\text{BC}}$ depends on how well $\boldsymbol{A}$
can be approximated by a circulant matrix. For Toeplitz-like matrices
in LMF, $\boldsymbol{C}$ can be constructed by minimizing $\Vert\boldsymbol{A}-\boldsymbol{C}\Vert_{F}$
\cite{chan1988optimal}. More generally, $\boldsymbol{C}$ can be
constructed to minimize $\Vert\boldsymbol{I}-\boldsymbol{C}^{-1}\boldsymbol{A}\Vert_{F}$
\cite{tyrtyshnikov1992optimal}. It can be seen that $\boldsymbol{\mathcal{M}}_{\text{BC}}$
is optimal as $h^{-2}\delta t\rightarrow0$. However, $\boldsymbol{\mathcal{M}}_{\text{BC}}$
requires complex arithmetic, and $\boldsymbol{\Lambda}$ in general
has $s$ distinct values. Hence, $\boldsymbol{\mathcal{M}}_{\text{BC}}$
is more expensive than BGS both in memory and computational cost per
KSP iteration. 

\subsubsection{Kronecker product splitting preconditioners}

More recently, H. Chen \cite{chen2014splitting} proposed the so-called
\emph{Kronecker product splitting} (\emph{KPS}) preconditioner, which
has the form
\begin{equation}
\boldsymbol{\mathcal{M}}_{\text{KPS}}(\alpha)=\nicefrac{1}{2\alpha}(\boldsymbol{I}+\alpha\boldsymbol{A})\otimes(\delta t\boldsymbol{K}+\alpha\boldsymbol{M})\label{eq:KPS}
\end{equation}
for some $\alpha>0$, where $\delta t\boldsymbol{K}+\alpha\boldsymbol{M}$
needs to be (approximately) factorized. $\boldsymbol{\mathcal{M}}_{\text{KPS}}$
is a splitting preconditioner in that $\boldsymbol{\mathcal{A}}=\boldsymbol{\mathcal{M}}_{\text{KPS}}(\alpha)-\nicefrac{1}{2\alpha}(\boldsymbol{I}-\alpha\boldsymbol{A})\otimes(\delta t\boldsymbol{K}-\alpha\boldsymbol{M})$.
Chen chose $\alpha$ by minimizing $\max_{\mu\in\lambda(\boldsymbol{A}^{-1})}\left|(\mu-\alpha)/(\mu+\alpha)\right|$,
where $\lambda(\boldsymbol{A}^{-1})$ denotes the eigenvalues of $\boldsymbol{A}^{-1}$.
From Proposition~\ref{prop:asymptotic-optimal}, it is easy to show
that $\boldsymbol{\mathcal{M}}_{\text{KPS}}$ is not optimal as $h^{-2}\delta t$
approaches $0$, neither is it optimal as $h^{-2}\delta t$ approaches
$\infty$; hence, we expect KPS to under-perform BGS and PNKP for
small and large $h^{-2}\delta t$, respectively. In \cite{chen2015generalized},
H. Chen proposed a \emph{generalized Kronecker product splitting}
(\emph{GKPS}) preconditioner 
\begin{equation}
\boldsymbol{\mathcal{M}}_{\text{GKPS}}(\alpha,\beta)=\nicefrac{1}{(\alpha+\beta)}(\boldsymbol{I}+\alpha\boldsymbol{A})\otimes(\delta t\boldsymbol{K}+\beta\boldsymbol{M})\label{eq:GKPS}
\end{equation}
based on the splitting $\boldsymbol{\mathcal{A}}=\boldsymbol{\mathcal{M}}_{\text{GKPS}}(\alpha,\beta)-\nicefrac{1}{(\alpha+\beta)}(\boldsymbol{I}-\beta\boldsymbol{A})\otimes(\delta t\boldsymbol{K}-\alpha\boldsymbol{M})$.
The optimal choice of $\alpha$ and $\beta$ requires the knowledge
of the spectrum of $\boldsymbol{M}^{-1}\boldsymbol{K}$ \cite{chen2015generalized}.
Both KPS and GKPS can be interpreted as \emph{NKP} (\emph{near Kronecker
product} \cite[Section 12.3.7]{Golub13MC}) approximations of $\mathcal{\boldsymbol{A}}$,
and they are cheaper than BGS in terms of factorization cost and memory
requirement. In terms of the number of GMRES iterations, KPS had mixed
performance compared to BGS \cite{chen2014splitting}, so did GKPS
\cite{chen2015generalized}. Since it is difficult to optimize the
parameters in GKPS, we will consider KPS instead of GKPS when comparing
with our methods.

\subsection{Single-level and multilevel incomplete factorization\label{subsec:Single-level-and-multilevel}}

The preconditioners mentioned above require (approximately) factorizing
the diagonal blocks. \emph{Incomplete LU }(\emph{ILU}) is arguably
one of the most promising for this purpose. Given a linear system
$\boldsymbol{\mathcal{A}}\mathcal{X}=\mathcal{B}$, ILU approximately
factorizes $\boldsymbol{\mathcal{A}}$ by
\begin{equation}
\boldsymbol{\mathcal{P}}^{T}\boldsymbol{\mathcal{A}}\boldsymbol{\mathcal{Q}}\approx\boldsymbol{\mathcal{L}}\boldsymbol{\mathcal{D}}\boldsymbol{\mathcal{U}},\label{eq:single-level-ILU}
\end{equation}
where $\boldsymbol{\mathcal{D}}$ is a diagonal matrix, and $\boldsymbol{\mathcal{L}}$
and $\boldsymbol{\mathcal{U}}$ are unit lower and upper triangular
matrices, respectively. The permutation matrices $\boldsymbol{\mathcal{P}}$
and $\boldsymbol{\mathcal{Q}}$ may be constructed statically (such
as using equilibration \cite{duff2001algorithms} or reordering \cite{amestoy1996approximate})
and dynamically (such as by pivoting \cite{saad1988preconditioning,saad2003iterative}).
We refer to \eqref{eq:single-level-ILU} as \emph{single-level ILU}.
The simplest form of single-level ILU is ILU($0$), which does not
have any dynamic pivoting and preserves the sparsity patterns of the
lower and upper triangular parts of $\boldsymbol{\mathcal{P}}^{T}\boldsymbol{\mathcal{A}}\boldsymbol{Q}$
in $\boldsymbol{\mathcal{L}}$ and $\boldsymbol{\mathcal{U}}$, respectively.
ILU(0) has linear time complexity in the number of nonzeros, and it
is often effective for fluid problems \cite{kanevsky2007application,pazner2017stage,persson2009scalable}.
For more challenging problems, ILU with dual thresholding (ILUT) \cite{saad1994ilut}
introduces \emph{fills} based on the levels in the elimination tree
and numerical values. One may also enable dynamic pivoting, leading
to so-called ILUP \cite{saad1988preconditioning} and ILUTP \cite{saad2003iterative}.
However, these more sophisticated variants often have superlinear
time complexity \cite{ghai2019comparison}, and they may suffer from
small pivots or unstable triangular factors \cite{saad2005multilevel}.

\emph{Multilevel incomplete LU} (\emph{MLILU}) is a general algebraic
framework for building block preconditioners. More precisely, a two-level
ILU reads
\begin{equation}
\boldsymbol{\mathcal{P}}^{T}\boldsymbol{\mathcal{A}}\boldsymbol{Q}=\begin{bmatrix}\boldsymbol{B} & \boldsymbol{F}\\
\boldsymbol{E} & \boldsymbol{C}
\end{bmatrix}\approx\boldsymbol{\mathcal{M}}=\begin{bmatrix}\tilde{\boldsymbol{B}} & \tilde{\boldsymbol{F}}\\
\tilde{\boldsymbol{E}} & \boldsymbol{C}
\end{bmatrix}=\begin{bmatrix}\boldsymbol{L} & \boldsymbol{0}\\
\boldsymbol{L}_{E} & \boldsymbol{I}
\end{bmatrix}\begin{bmatrix}\boldsymbol{D} & \boldsymbol{0}\\
\boldsymbol{0} & \boldsymbol{S}_{C}
\end{bmatrix}\begin{bmatrix}\boldsymbol{U} & \boldsymbol{U}_{F}\\
\boldsymbol{0} & \boldsymbol{I}
\end{bmatrix},\label{eq:two-level-ILU}
\end{equation}
where $\boldsymbol{B}\approx\tilde{\boldsymbol{B}}=\boldsymbol{L}\boldsymbol{D}\boldsymbol{U}$
corresponds to a single-level ILU of the leading block, and $\boldsymbol{S}_{C}=\boldsymbol{C}-\boldsymbol{L}_{E}\boldsymbol{D}\boldsymbol{U}_{F}$
is the Schur complement. One can also apply dynamic pivoting \cite{mayer2007multilevel}
or deferring \cite{bollhofer2006multilevel} in MLILU. For this two-level
ILU, $\boldsymbol{\mathcal{P}}\boldsymbol{\mathcal{M}}\boldsymbol{\mathcal{Q}}^{T}$
provides a preconditioner of $\boldsymbol{\mathcal{A}}$. By factorizing
$\boldsymbol{S}_{C}$ in \eqref{eq:two-level-ILU} recursively, we
obtain an MLILU and a corresponding multilevel preconditioner. The
recursion terminates when the Schur complement is sufficiently small,
and then a complete factorization can be employed. Compared to single-level
ILU, MLILU is generally more robust and effective, especially for
indefinite systems \cite{chen2021hilucsi,ghai2019comparison}.

In this work, we utilize a near-linear-complexity multilevel ILU called
HILUCSI, which stands for \emph{Hierarchical ILU-Crout with Scalability-oriented
and Inverse-based dropping} \cite{chen2021hilucsi}. HILUCSI shares
some similarities with other MLILU (such as ILUPACK \cite{bollhofer2011ilupack}).
However, its \emph{scalability-oriented dropping} enables near-linear
time complexity in its factorization and triangular solves in the
number of unknowns. We refer readers to \cite{chen2021hilucsi} for
details. The original implementation of HILUCSI in \cite{chen2021hilucsi}
only supported real arithmetic. For this work, we extended HILUCSI
to support complex arithmetic.

It is worth noting that ILU is decidedly challenging to achieve good
parallel scalability while maintaining its robustness and accuracy.
On massively parallel computers, a common practice is to utilize block
preconditioners in a domain-decomposition fashion (such as in PETSc
\cite{balay2019petsc}) or to use multigrid methods (such as in hypre
\cite{hypre2021}). These techniques often leverage incomplete factorization
(such as the supernodal ILU in SuperLU \cite{li2011supernodal} and
the parallel ILU in Euclid \cite{hysom2001scalable}) on the diagonal
blocks or as smoothers, which seem to strike a reasonable balance
between scalability and robustness.

It is also worth noting that one may attempt to apply a single-level
or multilevel ILU to the coefficient matrix $\boldsymbol{\mathcal{A}}$
in \eqref{eq:coeff_matrix} directly. However, the effectiveness of
virtually all the aforementioned preconditioners decreases as the
size of the matrices increases due to increased droppings. In addition,
it is computationally expensive to apply multilevel ILU to $\boldsymbol{\mathcal{A}}$
directly for many-stage FIRK. Therefore, it is desirable to take advantage
of the Kronecker-product structures, as we will demonstrate in Section~\ref{subsec:Comparison-of-near-linear}. 

\section{Mathematically Optimal Preconditioners\label{sec:Optimal-preconditioners}}

We now introduce two optimal preconditioners based on the Schur decomposition
and an alternative based on the Jordan form.

\subsection{Block complex Schur decomposition}

When $\delta t=\mathcal{O}(h^{\alpha})$ for $\alpha\approx1$, we
must consider both Kronecker products in \eqref{eq:coeff_matrix}.
To this end, we could consider a generalization of the block-circulant
preconditioners. In particular, it is well known that there is a complex
Schur decomposition (CSD) $\boldsymbol{A}=\boldsymbol{U}\boldsymbol{T}\boldsymbol{U}^{H}$
for any $\boldsymbol{A}\in\mathbb{C}^{s\times s}$, where $\boldsymbol{U}\in\mathbb{C}^{s\times s}$
is unitary and $\boldsymbol{T}\in\mathbb{C}^{s\times s}$ is upper
triangular. Hence, we define a \emph{block CSD} (\emph{BCSD}) preconditioner
analogous to $\boldsymbol{\mathcal{M}}_{\text{BC}}$, namely,
\[
\boldsymbol{\mathcal{M}}_{\text{BCSD}}=(\boldsymbol{U}\otimes\boldsymbol{I}_{n})\boldsymbol{\mathcal{T}}(\boldsymbol{U}^{H}\otimes\boldsymbol{I}_{n})\quad\text{with}\quad\boldsymbol{\mathcal{T}}=\boldsymbol{I}_{s}\otimes\boldsymbol{M}+\delta t\boldsymbol{T}\otimes\boldsymbol{K},
\]
where $\boldsymbol{\mathcal{T}}$ is block upper-triangular with $\mathbb{C}^{m\times m}$
diagonal blocks.
\begin{proposition}
\label{prop:BCSD} BCSD preconditioner is mathematically optimal for
$\boldsymbol{\mathcal{A}}$.
\end{proposition}

\begin{proof}
Using the product rule \eqref{eq:product-rule}, $\boldsymbol{\mathcal{M}}_{\text{BCSD}}=\boldsymbol{\mathcal{A}}$
with exact factorization.
\end{proof}
\begin{remark} $\boldsymbol{\mathcal{M}}_{\text{BCSD}}$ was motivated
by $\boldsymbol{\mathcal{M}}_{\text{BC}}$ in \eqref{eq:BC-1} and
it overcomes the approximation errors in the latter. $\boldsymbol{\mathcal{M}}_{\text{BC}}$
was originally designed for LMF, for which the matrix $\boldsymbol{I}_{s}$
in \eqref{eq:coeff_matrix} must be replaced by a general matrix $\boldsymbol{B}$.
Note that $\boldsymbol{\mathcal{M}}_{\text{BCSD}}$ can be generalized
to LMF by using a generalized CSD \cite[Section 7.7.2]{Golub13MC}
on $\boldsymbol{A}$ and $\boldsymbol{B}$.\end{remark}

In practice, we need to factorize the diagonal blocks in $\boldsymbol{\mathcal{T}}$
approximately. In this case, $\boldsymbol{\mathcal{M}}_{\text{BCSD}}$
is $\epsilon$-accurate with a sufficiently accurate approximate factorization.
\begin{theorem}
Given $\tilde{\boldsymbol{\mathcal{T}}}=\boldsymbol{\mathcal{T}}+\delta\boldsymbol{\mathcal{T}}$
such that $\left\Vert \tilde{\boldsymbol{\mathcal{T}}}^{-1}\right\Vert \leq C$
for some bounded $C>0$, $\tilde{\boldsymbol{\mathcal{M}}}=(\boldsymbol{U}\otimes\boldsymbol{I}_{n})\tilde{\boldsymbol{\mathcal{T}}}(\boldsymbol{U}^{H}\otimes\boldsymbol{I}_{n})$
is an $\epsilon$-accurate preconditioner with sufficiently small
$\left\Vert \delta\boldsymbol{\mathcal{T}}\right\Vert $.
\end{theorem}

\begin{proof}
Let $\boldsymbol{\mathcal{X}}=\boldsymbol{U}\otimes\boldsymbol{I}_{n}$
in Definition~\ref{def:epsilon-accuracy}. Then, $\boldsymbol{\mathcal{X}}^{-1}=\boldsymbol{U}^{H}\otimes\boldsymbol{I}_{n}$
and 
\begin{align*}
\left\Vert \boldsymbol{\mathcal{X}}^{-1}\left(\boldsymbol{\mathcal{A}}\tilde{\boldsymbol{\mathcal{M}}}^{-1}\right)\boldsymbol{\boldsymbol{\mathcal{X}}}-\boldsymbol{I}\right\Vert  & =\left\Vert \boldsymbol{\mathcal{X}}^{-1}\left(\boldsymbol{\mathcal{M}}_{\text{BCSD}}\tilde{\boldsymbol{\mathcal{M}}}^{-1}\right)\boldsymbol{\mathcal{X}}-\boldsymbol{I}\right\Vert \\
 & =\left\Vert \boldsymbol{\mathcal{T}}\tilde{\boldsymbol{\mathcal{T}}}^{-1}-\boldsymbol{I}\right\Vert \\
 & =\left\Vert \delta\boldsymbol{\mathcal{T}}\tilde{\boldsymbol{\mathcal{T}}}^{-1}\right\Vert ,
\end{align*}
which is bounded by $\left\Vert \delta\boldsymbol{\mathcal{T}}\right\Vert C$.
Hence, $\tilde{\boldsymbol{\mathcal{M}}}$ is $\epsilon$-accurate
with a bounded $C$ and sufficiently small $\left\Vert \delta\boldsymbol{\mathcal{T}}\right\Vert $.
\end{proof}
When using HILUCSI to factorize its diagonal blocks, the boundedness
of $\left\Vert \tilde{\boldsymbol{\mathcal{T}}}^{-1}\right\Vert $
is ensured since HILUCSI monitors and dynamically controls the condition
numbers of the triangular and diagonal factors using both static and
dynamic permutations. For sufficiently tight dropping thresholds in
HILUCSI, $\Vert\delta\boldsymbol{\mathcal{T}}\Vert\ll1/\left\Vert \tilde{\boldsymbol{\mathcal{T}}}^{-1}\right\Vert $.
In addition, $\Vert\delta\boldsymbol{\mathcal{T}}\Vert$ tends to
be smaller for small $h^{-2}\delta t$ due to the (block) diagonal
dominance of $\boldsymbol{I}_{s}\otimes\boldsymbol{M}$.

\subsection{Ordering of BCSD}

CSD is not unique, and any ordering of the eigenvalues in its diagonal
entries suffices in exact arithmetic. With incomplete factorization
of the diagonal blocks, different decompositions may lead to different
convergence rates of GMRES. To optimize the ordering, we derive an
error analysis for BCSD as follows. 
\begin{lemma}
\label{lem:error-equality}Suppose that the diagonal blocks $\boldsymbol{D}_{i}\in\mathbb{C}^{m\times m}$
of $\boldsymbol{\mathcal{T}}\in\mathbb{C}^{sm\times sm}$ are approximated
by $\tilde{\boldsymbol{D}}_{i}=\boldsymbol{D}_{i}+\delta\boldsymbol{D}_{i}$,
while the off-diagonal blocks $\boldsymbol{T}_{ij}$ are exact. Let
$\delta\boldsymbol{\mathcal{D}}$ be the block diagonal matrix composed
of $\delta\boldsymbol{D}_{i}$, $\tilde{\boldsymbol{\mathcal{T}}}=\boldsymbol{\mathcal{T}}+\delta\boldsymbol{\mathcal{D}}$,
and $\tilde{\boldsymbol{\mathcal{M}}}=(\boldsymbol{U}\otimes\boldsymbol{I}_{m})\tilde{\boldsymbol{\mathcal{T}}}(\boldsymbol{U}^{H}\otimes\boldsymbol{I}_{m})$.
Given a vector $\mathcal{B}\in\mathbb{C}^{m}$, let $\mathcal{Y}=(\boldsymbol{U}^{H}\otimes\boldsymbol{I}_{m})\boldsymbol{\mathcal{M}}_{\text{BCSD}}^{-1}\mathcal{B}$
and $\tilde{\mathcal{Y}}=\mathcal{\mathcal{Y}}+\mathcal{\delta\mathcal{Y}}=(\boldsymbol{U}^{H}\otimes\boldsymbol{I}_{m})\tilde{\boldsymbol{\mathcal{M}}}^{-1}\mathcal{B}$.
Then, 
\begin{equation}
\boldsymbol{\mathcal{T}}\mathcal{\delta Y}=-\delta\boldsymbol{\mathcal{D}}\tilde{\mathcal{Y}}.\label{eq:error-equality}
\end{equation}
\end{lemma}

For convenience, let $\boldsymbol{y}_{i}$, $\tilde{\boldsymbol{y}}_{i}$
and $\delta\boldsymbol{y}_{i}$ denote the $i$th block in $\mathcal{Y}$,
$\tilde{\mathcal{Y}}$, and $\delta\mathcal{Y}$ corresponding to
$\boldsymbol{D}_{i}$, respectively.
\begin{proof}
Let $\hat{\boldsymbol{b}}_{i}$ denote the $i$th block in $(\boldsymbol{U}^{H}\otimes\boldsymbol{I}_{m})\mathcal{B}$.
By definition,
\[
\tilde{\boldsymbol{D}}_{i}\tilde{\boldsymbol{y}}_{i}+{\textstyle \sum}_{j=i+1}^{s}\boldsymbol{T}_{ij}\tilde{\boldsymbol{y}}_{j}=\hat{\boldsymbol{b}}_{i}=\boldsymbol{D}_{i}\boldsymbol{y}_{i}+{\textstyle \sum}_{j=i+1}^{s}\boldsymbol{T}_{ij}\boldsymbol{y}_{j}.
\]
Hence, $\boldsymbol{D}_{i}\delta\boldsymbol{y}_{i}+\sum_{j=i+1}^{s}\boldsymbol{T}_{ij}\delta\boldsymbol{y}_{j}=-\delta\boldsymbol{D}_{i}\tilde{\boldsymbol{y}}_{i},$
which is equivalent to \eqref{eq:error-equality}.
\end{proof}
Note that $\left\Vert \boldsymbol{\mathcal{T}}^{-1}\right\Vert $
is invariant of the ordering. From a standard norm-wise error analysis
$\left\Vert \mathcal{\delta Y}\right\Vert \leq\left\Vert \boldsymbol{\mathcal{T}}^{-1}\right\Vert \Vert\delta\boldsymbol{\mathcal{D}}\tilde{\mathcal{Y}}\Vert$,
one may conclude that $\left\Vert \mathcal{\delta Y}\right\Vert $
is insensitive to the ordering, assuming $\Vert\tilde{\mathcal{Y}}\Vert\approx\Vert\mathcal{Y}\Vert$.
However, the following asymptotic analysis suggests that the component-wise
errors do depend on the ordering of CSD and may be optimized under
some reasonable assumptions. 
\begin{proposition}
\label{prop:error-bnd}Assume \textup{\emph{$\left\Vert \delta\boldsymbol{\mathcal{D}}\right\Vert =\mathcal{O}(\varepsilon)$
for some small $\varepsilon$ and $\left\Vert \boldsymbol{D}_{i}^{-1}\right\Vert =\mathcal{O}(1)$.
Then,}}\textup{
\begin{align}
\Vert\delta\boldsymbol{y}_{i}\Vert & \leq\left\Vert \boldsymbol{D}_{i}^{-1}\right\Vert \left(\left\Vert \delta\boldsymbol{D}_{i}\boldsymbol{y}_{i}\right\Vert +{\textstyle \sum}_{j=i+1}^{s}\left\Vert \boldsymbol{T}_{ij}\right\Vert \left\Vert \delta\boldsymbol{y}_{j}\right\Vert \right)+\mathcal{O}(\varepsilon^{2}).\label{eq:error-amplication-factor}
\end{align}
}
\end{proposition}

\begin{proof}
Due to Lemma~\ref{lem:error-equality},
\begin{equation}
\boldsymbol{D}_{i}\delta\boldsymbol{y}_{i}=-\sum_{j=i+1}^{s}\boldsymbol{T}_{ij}\delta\boldsymbol{y}_{j}-\delta\boldsymbol{D}_{i}\tilde{\boldsymbol{y}}_{i}=-\sum_{j=i+1}^{s}\boldsymbol{T}_{ij}\delta\boldsymbol{y}_{j}-\delta\boldsymbol{D}_{i}\boldsymbol{y}_{i}-\delta\boldsymbol{D}_{i}\delta\boldsymbol{y}_{i}.\label{eq:error-recursion}
\end{equation}
Expanding the recursion, we conclude that $\Vert\delta\boldsymbol{y}_{i}\Vert=\mathcal{O}(\left\Vert \delta\boldsymbol{\mathcal{D}}\right\Vert )=\mathcal{O}(\varepsilon)$
for a constant $s$. Hence, $\left\Vert \delta\boldsymbol{D}_{i}\delta\boldsymbol{y}_{i}\right\Vert \leq\left\Vert \delta\boldsymbol{D}_{i}\right\Vert \left\Vert \delta\boldsymbol{y}_{i}\right\Vert =\mathcal{O}(\varepsilon^{2})$,
and \eqref{eq:error-amplication-factor} then follows from \eqref{eq:error-recursion}.
\end{proof}
Let $d_{i}$ denote the $i$th diagonal block in $\boldsymbol{T}$
in CSD. Note that as $h^{-2}\delta t$ approaches $\infty$, $\left\Vert \boldsymbol{D}_{i}^{-1}\right\Vert $
tends to $\vert d_{i}^{-1}\vert\left\Vert \boldsymbol{K}^{-1}\right\Vert $.
For FDM, $\left\Vert \boldsymbol{K}^{-1}\right\Vert =\mathcal{O}(1)$
(due to a similar argument as Lemma~\ref{lem:mass-stiff-relationship}),
so the assumption of $\left\Vert \boldsymbol{D}_{i}^{-1}\right\Vert =\mathcal{O}(1)$
is reasonable; for FEM, $\left\Vert \boldsymbol{K}^{-1}\right\Vert =\mathcal{O}(1)\left\Vert \boldsymbol{M}^{-1}\right\Vert $
\cite{ern2013theory}, so we need to generalize the assumptions to
take into account $\left\Vert \boldsymbol{M}^{-1}\right\Vert $. In
either case, following a similar procedure as in the backward error
analysis for back solve \cite[pp. 122--127]{trefethen1997numerical},
we can conclude that the recursions in \eqref{eq:error-amplication-factor}
would lead to an amplification factor of $\prod_{k=j}^{s}\left\Vert \boldsymbol{D}_{k}^{-1}\right\Vert \left\Vert \boldsymbol{T}_{jk}\right\Vert $
on $\left\Vert \delta\boldsymbol{D}_{j}\boldsymbol{y}_{j}\right\Vert $
for $i+1\leq j\leq s$ in $\Vert\delta\boldsymbol{y}_{i}\Vert$. Assuming
$\vert d_{i}^{-1}\vert>1$ and $\left\Vert \delta\boldsymbol{D}_{j}\boldsymbol{y}_{j}\right\Vert $
and $\left\Vert \boldsymbol{T}_{ij}\right\Vert $ are insensitive
to reordering, we can decrease these amplification factors by making
$\vert d_{i}^{-1}\vert=1/\vert d_{i}\vert$ as small as possible for
large $i$ in $\boldsymbol{T}$. Hence, we order the CSD so that $\vert d_{i}\vert$
is in ascending order. We will demonstrate the benefit of this ordering
for the five- and six-stage GL schemes in Section~\ref{subsec:Effect-of-ordering}.
To compute this ordering, one can use the MATLAB code in \cite{brandts2002matlab}
to sort the real Schur decomposition (RSD) based on the complex eigenvalues,
and then convert the RSD to CSD using the MATLAB function \textsf{rsf2csf}
\cite{MATLAB2020b}.

\subsection{Block real Schur decomposition}

For IRK, $\boldsymbol{A}\in\mathbb{R}^{s\times s}$, and there is
a \emph{real Schur decomposition} (\emph{RSD}) \cite[Section 7.4.1]{Golub13MC},
\begin{equation}
\boldsymbol{A}=\boldsymbol{Q}\boldsymbol{R}\boldsymbol{Q}^{T},\label{eq:RSD}
\end{equation}
where $\boldsymbol{Q}\in\mathbb{R}^{s\times s}$ is orthogonal (i.e.,
$\boldsymbol{Q}\boldsymbol{Q}^{T}=\boldsymbol{I}$) and $\boldsymbol{R}\in\mathbb{R}^{s\times s}$
is \emph{quasi-triangular}, with 1-by-1 and 2-by-2 diagonal blocks.
We define a \emph{block RSD} (\emph{BRSD}) preconditioner as
\begin{equation}
\boldsymbol{\mathcal{M}}_{\text{BRSD}}=(\boldsymbol{Q}\otimes\boldsymbol{I}_{n})\boldsymbol{\mathcal{R}}(\boldsymbol{Q}^{T}\otimes\boldsymbol{I}_{n})\quad\text{with}\quad\boldsymbol{\mathcal{R}}=\boldsymbol{I}_{s}\otimes\boldsymbol{M}+\delta t\boldsymbol{R}\otimes\boldsymbol{K},\label{eq:BRSD}
\end{equation}
where $\boldsymbol{\mathcal{R}}$ is block quasi-triangular with $\mathbb{R}^{m\times m}$
and $\mathbb{R}^{2m\times2m}$ diagonal blocks.
\begin{proposition}
\label{prop:BRSD}BRSD preconditioner is mathematically optimal for
$\boldsymbol{\mathcal{A}}$.
\end{proposition}

\begin{proof}
It follows from the same argument as Proposition~\ref{prop:BCSD}. 
\end{proof}
Like $\boldsymbol{\mathcal{M}}_{\text{BCSD}}$, $\boldsymbol{\mathcal{M}}_{\text{BRSD}}$
can also be generalized to LMF by using a generalized RSD \cite[Section 7.7.2]{Golub13MC}.
BRSD requires only real arithmetic. However, its larger $2m\times2m$
blocks are more expensive to factorize than the $m\times m$ complex-valued
blocks in BCSD. In addition, the larger blocks in BRSD may also decrease
its effectiveness compared to BCSD, as we will demonstrate in Section~\ref{subsec:Results-optimal-preconditioners}.

Like CSD, RSD is not unique, and different ordering may lead to different
convergence rate. We can generalize the analysis in Proposition~\ref{prop:error-bnd}
and then assume $\left\Vert \delta\boldsymbol{D}_{j}\boldsymbol{x}_{j}\right\Vert $
and $\left\Vert \boldsymbol{R}_{ij}\right\Vert $ are insensitive
to the ordering for the diagonal blocks $\boldsymbol{D}_{i}\in\mathbb{R}^{2m\times2m}\cup\mathbb{R}^{m\times m}$
and off-diagonal blocks $\boldsymbol{R}_{ij}\in\mathbb{R}^{2m\times2m}\cup\mathbb{R}^{m\times m}$
of $\boldsymbol{\mathcal{R}}\in\mathbb{R}^{sm\times sm}$. Such an
assumption seems reasonable for even-stage GL. For odd-stage GL, we
put the 1-by-1 diagonal block corresponding to the real eigenvalue
at the upper-left corner of $\boldsymbol{R}$ in \eqref{eq:RSD},
and then a similar assumption seems reasonable to the other blocks.
Our numerical experiments show that such an ordering indeed improves
the effectiveness of BRSD for GL schemes. In addition, we observed
that this ordering for GL schemes coincides with sorting the real
parts of the eigenvalues of RSD in descending order. For completeness,
we describe the procedure in Appendix~\ref{sec:sort-permute-RSD}.
 As examples, Table~\ref{tab:sorted-RSD} shows the $\boldsymbol{Q}$
and $\boldsymbol{R}$ matrices of BRSD for the three- and four-stage
GL schemes; for the two-stage GL, $\boldsymbol{Q}=\boldsymbol{I}_{2}$
and $\boldsymbol{R}=\boldsymbol{A}$.

\begin{table}
\caption{\label{tab:sorted-RSD}The $\boldsymbol{Q}$ and $\boldsymbol{R}$
matrices for BRSD from a sorted and permuted real Schur forms of the
Butcher matrices for the three and four-stage GL schemes.}

\centering{}\setlength\tabcolsep{2pt}{\footnotesize{}}%
\begin{tabular}{cccc}
\toprule 
$s$ &  & BRSD ($\boldsymbol{Q}$) & BRSD ($\boldsymbol{R}$)\tabularnewline
\midrule 
3 &  & $\begin{bmatrix}0.0715 & -0.1246 & -0.9896\\
0.1177 & 0.9863 & -0.1157\\
0.9905 & -0.1082 & 0.0852
\end{bmatrix}$ & $\begin{bmatrix}0.2153 & 0.4392 & -0.3537\\
0 & 0.1423 & -0.2704\\
0 & 0.0682 & 0.1423
\end{bmatrix}$\tabularnewline
\midrule 
4 &  & {\small{}$\begin{bmatrix}-0.0117 & 0.2096 & 0.3219 & 0.9232\\
0.0283 & 0.1012 & -0.9458 & 0.3072\\
0.1826 & 0.9561 & 0.0342 & -0.2267\\
0.9827 & -0.1780 & 0.0247 & 0.0442
\end{bmatrix}$} & {\small{}$\begin{bmatrix}0.1584 & 0.4262 & -0.2749 & 0.2277\\
-0.0053 & 0.1584 & -0.2219 & 0.2206\\
0 & 0 & 0.0916 & -0.1834\\
0 & 0 & 0.0729 & 0.0916
\end{bmatrix}$}\tabularnewline
\bottomrule
\end{tabular}{\footnotesize\par}
\end{table}

\subsection{Block Jordan form}

As an alternative to BCSD and BRSD, one might utilize a Jordan decomposition
$\boldsymbol{A}=\boldsymbol{X}\boldsymbol{\Lambda}\boldsymbol{X}^{-1}$
to construct a \emph{block Jordan form} (\emph{BJF}) preconditioner
\begin{equation}
\boldsymbol{\mathcal{M}}_{\text{BJF}}=(\boldsymbol{X}\otimes\boldsymbol{I}_{n})(\boldsymbol{I}_{s}\otimes\boldsymbol{M}+\delta t\boldsymbol{\Lambda}\otimes\boldsymbol{K})(\boldsymbol{X}^{-1}\otimes\boldsymbol{I}_{n}),\label{eq:BJF}
\end{equation}
where $\boldsymbol{X}$ and $\boldsymbol{\Lambda}$ are in $\mathbb{C}^{s\times s}$.
If $\boldsymbol{A}$ is nondefective, as in GL schemes, then $\boldsymbol{\Lambda}$
is diagonal and contains (complex) eigenvalues. This block Jordan
form was developed as a solution technique for IRK independently by
Butcher \cite{butcher1976implementation} and Bickart \cite{bickart1977efficient};
see also \cite[p. 122]{hairer1996solving}. It was also leveraged
recently in \cite{huang2019conditioning} to analyze the condition
numbers of IRK in conjunction with diagonal preconditioners. To the
best of our knowledge, BJF has not been used as a preconditioner for
IRK in the literature. Hence, we consider it as a new block-preconditioner
motivated by Butcher \cite{butcher1976implementation} and Bickart
\cite{bickart1977efficient}. Like BCSD, BJF also requires complex
arithmetic, and it is mathematically optimal in exact arithmetic.
However, BJF cannot be generalized to LMF. More importantly, BJF is
not $\epsilon$-accurate with incomplete factorization for GL-schemes
with large $s$, because $\boldsymbol{X}$ in \eqref{eq:BJF} is not
orthogonal in general, and there is empirical evidence that $\kappa(\boldsymbol{X})$
grows exponentially in $s$. Nevertheless, it is worth considering
BJF for small $s$.

\section{Near-Optimal Preconditioners\label{sec:Near-Optimal-Preconditioners}}

We now introduce a ``singly diagonal'' approximation of BRSD with
reduced cost. We also compare it qualitatively with others.

\subsection{\label{subsec:BARSD}Singly-diagonal approximate BRSD}

We now describe a \emph{singly-diagonal approximate BRSD} (\emph{SABRSD})
preconditioner, 
\begin{equation}
\boldsymbol{\mathcal{M}}_{\text{SABRSD}}=(\boldsymbol{Q}\otimes\boldsymbol{I}_{n})\hat{\boldsymbol{\mathcal{R}}}(\boldsymbol{Q}^{T}\otimes\boldsymbol{I}_{n}),\text{ where }\hat{\boldsymbol{\mathcal{R}}}=\boldsymbol{I}_{s}\otimes\boldsymbol{M}+\delta t\hat{\boldsymbol{R}}\otimes\boldsymbol{K},\label{eq:SBARSD}
\end{equation}
where $\hat{\boldsymbol{R}}$ is a \emph{singly diagonally upper triangular}
(\emph{SDUT}) matrix, i.e., a scalar multiple of a unit upper triangular
matrix. Let $\mathbb{T}_{s}$ denote the set of $s\times s$ SDUT
matrices, i.e., $\mathbb{T}_{s}=\{\boldsymbol{T}\in\mathbb{R}^{s\times s}\mid t_{ii}=t_{jj}\wedge t_{ij}=0,\forall j<i\}$.
We compute $\hat{\boldsymbol{R}}$ to approximate $\boldsymbol{R}$
in \eqref{eq:RSD} using a constrained minimization 
\begin{equation}
\hat{\boldsymbol{R}}=\arg\min_{\hat{\boldsymbol{R}}\in\mathbb{T}_{s}}\left\Vert \boldsymbol{I}-\boldsymbol{R}\hat{\boldsymbol{R}}^{-1}\right\Vert \quad\text{ subject to }\quad\min_{\hat{\boldsymbol{R}}\in\mathbb{T}_{s}}\kappa(\boldsymbol{R}\hat{\boldsymbol{R}}^{-1}).\label{eq:min-norm}
\end{equation}

\begin{remark}

The objective function in \eqref{eq:min-norm} shares some similarity
with that in \cite{tyrtyshnikov1992optimal} for the ``superoptimal''
condition for a circulant left-preconditioner $\boldsymbol{C}$, where
Tyrtyshnikov minimized $\left\Vert \boldsymbol{I}-\boldsymbol{C}^{-1}\boldsymbol{A}\right\Vert _{F}$
without an analogous constraint on $\kappa(\boldsymbol{C}^{-1}\boldsymbol{A})$.
Our constraint on $\kappa(\boldsymbol{R}\hat{\boldsymbol{R}}^{-1})$
is similar to the objective function used by Staff et al. \cite{staff2006Spreconditioning}
in optimizing a block lower-triangular preconditioner (analogous to
$\hat{\boldsymbol{R}}^{T}$ if we replace $\boldsymbol{Q}$ with $\boldsymbol{I}_{s}$),
where they set the diagonal entries of $\hat{\boldsymbol{R}}^{T}$
to be those in $\boldsymbol{A}$ without taking into account $\left\Vert \boldsymbol{I}-\boldsymbol{R}\hat{\boldsymbol{R}}^{-1}\right\Vert $.

\end{remark}

The dual conditions in \eqref{eq:min-norm} were motivated our numerical
experimentation; see Section~\ref{subsec:Effect-of-singly}. We can
also derive them based on Definition~\ref{def:epsilon-accuracy}.
Let us assume $\delta t\gtrsim\mathcal{O}(h)$, so $h^{-2}\delta t\gg1$.
In this case, $\boldsymbol{K}$ dominates $\boldsymbol{M}$, so it
is reasonable to absorb the perturbations due to the approximate factorization
into $\boldsymbol{K}$, so that $\boldsymbol{I}_{s}\otimes\boldsymbol{M}+\delta t\boldsymbol{R}\otimes\boldsymbol{K}$
is approximated by $\boldsymbol{I}_{s}\otimes\boldsymbol{M}+\delta t\hat{\boldsymbol{R}}\otimes\tilde{\boldsymbol{K}}$,
where $\hat{\boldsymbol{R}}=\boldsymbol{R}+\delta\boldsymbol{R}$
and $\tilde{\boldsymbol{K}}=\boldsymbol{K}+\delta\boldsymbol{K}$.
Hence, $\boldsymbol{I}_{s}\otimes\boldsymbol{M}+\delta t\boldsymbol{R}\otimes\boldsymbol{K}\approx\delta t\boldsymbol{R}\otimes\boldsymbol{K}$
and $\boldsymbol{I}_{s}\otimes\boldsymbol{M}+\delta t\hat{\boldsymbol{R}}\otimes\tilde{\boldsymbol{K}}\approx\delta t\hat{\boldsymbol{R}}\otimes\tilde{\boldsymbol{K}}$.
Then,
\begin{align*}
 & \Vert(\boldsymbol{Q}\otimes\boldsymbol{I})\left(\boldsymbol{I}_{s}\otimes\boldsymbol{M}+\delta t\boldsymbol{R}\otimes\boldsymbol{K}\right)(\boldsymbol{I}_{s}\otimes\boldsymbol{M}+\delta t\hat{\boldsymbol{R}}\otimes\tilde{\boldsymbol{K}})^{-1}(\boldsymbol{Q}^{T}\otimes\boldsymbol{I})-\boldsymbol{I}\Vert\\
\approx & \left\Vert (\boldsymbol{R}\otimes\boldsymbol{K})(\hat{\boldsymbol{R}}\otimes\tilde{\boldsymbol{K}})^{-1}-\boldsymbol{I}\right\Vert \\
= & \left\Vert \left((\hat{\boldsymbol{R}}-\delta\boldsymbol{R})\hat{\boldsymbol{R}}^{-1}\right)\otimes\left((\tilde{\boldsymbol{K}}-\delta\boldsymbol{K})\tilde{\boldsymbol{K}}^{-1}\right)-\boldsymbol{I}\right\Vert \\
= & \left\Vert \delta\boldsymbol{R}\hat{\boldsymbol{R}}^{-1}\otimes\boldsymbol{I}_{m}+\boldsymbol{R}\hat{\boldsymbol{R}}^{-1}\otimes\left(\delta\boldsymbol{K}\tilde{\boldsymbol{K}}^{-1}\right)\right\Vert \\
\leq & \left\Vert \boldsymbol{I}-\boldsymbol{R}\hat{\boldsymbol{R}}^{-1}\right\Vert +\left\Vert \boldsymbol{R}\hat{\boldsymbol{R}}^{-1}\right\Vert \otimes\left\Vert \delta\boldsymbol{K}\tilde{\boldsymbol{K}}^{-1}\right\Vert .
\end{align*}
In \eqref{eq:min-norm}, the constraint $\min\kappa(\boldsymbol{R}\hat{\boldsymbol{R}}^{-1})$,
where $\kappa(\boldsymbol{R}\hat{\boldsymbol{R}}^{-1})=\left\Vert \hat{\boldsymbol{R}}\boldsymbol{R}^{-1}\right\Vert \left\Vert \boldsymbol{R}\hat{\boldsymbol{R}}^{-1}\right\Vert $,
ensures a bounded $\left\Vert \boldsymbol{R}\hat{\boldsymbol{R}}^{-1}\right\Vert $,
assuming $\left\Vert \hat{\boldsymbol{R}}\boldsymbol{R}^{-1}\right\Vert $
is nearly a constant. Note that $\kappa(\boldsymbol{R}\hat{\boldsymbol{R}}^{-1})$
is invariant under scaling of $\hat{\boldsymbol{R}}$. The minimization
of $\left\Vert \boldsymbol{I}-\boldsymbol{R}\hat{\boldsymbol{R}}^{-1}\right\Vert $
determines the scaling factor of $\hat{\boldsymbol{R}}$. Due to its
singly-diagonal property, SABRSD has comparable cost as PNKP, assuming
the overhead in multiplying vectors with $\boldsymbol{Q}$ and $\boldsymbol{Q}^{T}$
is negligible. 

\subsection{Ordering and permuting RSD for SABRSD\label{subsec:Ordering-and-permuting}}

Similar to BRSD, the ordering of the eigenvalues along the diagonal
blocks in RSD can affect the effectiveness of SABRSD. To derive an
optimal ordering, assume the lower-triangular part of $\boldsymbol{R}$
is negligible. In this case, Proposition~\ref{prop:error-bnd} suggests
that it is desirable to sort the diagonal entries in $\boldsymbol{R}$
in ascending order, in contrast to descending order in $\boldsymbol{\mathcal{M}}_{\text{BRSD}}$.
To make the lower-triangular part negligible, we make the 2-by-2 diagonal
blocks in $\boldsymbol{R}$ as upper-triangularly dominant as possible,
by making use of the following fact.
\begin{proposition}
\label{prop:min-lower-tri}Let $\boldsymbol{D}=\begin{bmatrix}\alpha & \gamma\\
\beta & \alpha
\end{bmatrix}$ with $\alpha>0$ and $\beta\gamma<0$, and let $\boldsymbol{S}$
be an orthogonal matrix. The lower-triangular component in $\tilde{\boldsymbol{D}}=\boldsymbol{S}\boldsymbol{D}\boldsymbol{S}^{T}$
attains the minimal magnitude when $\boldsymbol{S}=\boldsymbol{I}_{2}$
and $\boldsymbol{S}=\begin{bmatrix}0 & 1\\
1 & 0
\end{bmatrix}$ for $\vert\beta\vert\leq\vert\gamma\vert$ and $\vert\beta\vert>\vert\gamma\vert$
, respectively.
\end{proposition}

\begin{proof}
Consider rotation $\boldsymbol{S}_{\theta}=\begin{bmatrix}\cos\theta & \sin\theta\\
-\sin\theta & \cos\theta
\end{bmatrix}$, and then $\tilde{d}_{21}=\gamma-(\beta+\gamma)\sin(2\theta)$. $\left|\tilde{d}_{21}\right|$
is minimized when $\theta=0$ and $\theta=\pi/2$ for $\vert\beta\vert\leq\vert\gamma\vert$
and $\vert\beta\vert>\vert\gamma\vert$, respectively. For $\theta=\pi/2$,
$\boldsymbol{S}=\begin{bmatrix}0 & 1\\
1 & 0
\end{bmatrix}$ flips the sign of the second row in $\boldsymbol{S}_{\theta}$.
\end{proof}
 In Proposition~\ref{prop:min-lower-tri}, $\beta\gamma<0$ is satisfied
for any 2-by-2 diagonal block in RSD; the assumption of $\alpha>0$
is equivalent to requiring the real parts for the eigenvalues of the
Butcher matrix $\boldsymbol{A}$ to be positive, which is valid for
A-stable FIRK with an invertible $\boldsymbol{A}$ \cite[p. 402]{hairer1996solving}.
We will demonstrate the benefit of this ordering strategy in Section~\ref{subsec:Effect-of-ordering}.
 Given the resulting RSD, $\hat{\boldsymbol{R}}$ in \eqref{eq:SBARSD}
can be obtained by using the optimization procedure  described in
Appendix~\ref{sec:SABRSD-code}. Table~\ref{tab:SBARSD-SOBT} gives
the $\boldsymbol{Q}$ and $\hat{\boldsymbol{R}}$ matrices for the
two-, three- and four-stage GL schemes.

\begin{table}
\caption{\label{tab:SBARSD-SOBT}The $\boldsymbol{Q}$ and $\hat{\boldsymbol{R}}$
matrices of SABRSD for GL schemes with two to four stages.}

\centering{}\setlength\tabcolsep{2pt}%
\begin{tabular}{cccc}
\toprule 
$s$ &  & SABRSD ($\boldsymbol{Q}$) & SABRSD ($\hat{\boldsymbol{R}}$)\tabularnewline
\midrule
2 &  & $\begin{bmatrix}0 & 1\\
1 & 0
\end{bmatrix}$ & $\begin{bmatrix}0.3008 & 0.5876\\
0 & 0.3008
\end{bmatrix}$\tabularnewline
\midrule 
3 &  & $\begin{bmatrix}0.0832 & -0.1811 & -0.9799\\
-0.0603 & 0.9806 & -0.1863\\
-0.9947 & -0.0746 & -0.0707
\end{bmatrix}$ & $\begin{bmatrix}0.2339 & -0.5739 & 0.4284\\
0 & 0.2339 & -0.3351\\
0 & 0 & 0.2339
\end{bmatrix}$\tabularnewline
\midrule 
4 &  & {\small{}$\begin{bmatrix}-0.0482 & -0.0225 & 0.3367 & 0.9401\\
0.1222 & -0.0925 & -0.9290 & 0.3368\\
-0.1117 & 0.9879 & -0.0943 & 0.0517\\
-0.9850 & -0.1224 & -0.1211 & -0.0101
\end{bmatrix}$} & {\small{}$\begin{bmatrix}0.212 & -0.535 & 0.3444 & -0.3663\\
0 & 0.212 & -0.3338 & 0.3178\\
0 & 0 & 0.212 & -0.274\\
0 & 0 & 0 & 0.212
\end{bmatrix}$}\tabularnewline
\bottomrule
\end{tabular}
\end{table}

\subsection{Alternative approximations to BRSD}

It is well known that the Butcher matrix $\boldsymbol{A}$ is strongly
lower-triangular dominant in GL schemes \cite{van1997triangularly}.
Hence, one may fix $\boldsymbol{Q}$ in SABRSD to be the ``flipped''
identity matrix $\boldsymbol{P}=[\boldsymbol{e}_{s},\boldsymbol{e}_{s-1},\dots,\boldsymbol{e}_{1}]$,
so that $\boldsymbol{P}\boldsymbol{A}\boldsymbol{P}^{T}$ is nearly
upper triangular. Let $\hat{\boldsymbol{L}}=\boldsymbol{P}\hat{\boldsymbol{R}}\boldsymbol{P}$
minimize $\Vert\boldsymbol{I}-\boldsymbol{P}\boldsymbol{A}\boldsymbol{P}\hat{\boldsymbol{R}}^{-1}\Vert=\Vert\boldsymbol{I}-\boldsymbol{A}\hat{\boldsymbol{L}}^{-1}\Vert$
subject to the diagonal entries in $\hat{\boldsymbol{L}}$ being equal.
We then obtain a simple block preconditioner
\begin{equation}
\boldsymbol{\mathcal{M}}_{\text{SOBT}}=\boldsymbol{I}_{s}\otimes\boldsymbol{M}+\delta t\hat{\boldsymbol{L}}\otimes\boldsymbol{K},\label{eq:OBT-precond}
\end{equation}
which we refer to as a \emph{singly-diagonal optimized block triangular}
(\emph{SOBT}) preconditioner. SOBT is similar to the optimized block
triangular preconditioners in \cite{staff2006Spreconditioning}, except
that we use a different constrained minimization \eqref{eq:min-norm}.
One can compute SOBT by modifying the procedure described in Appendix~\ref{sec:SABRSD-code}
to minimize $\left\Vert \boldsymbol{I}-\boldsymbol{A}\hat{\boldsymbol{L}}^{-1}\right\Vert $.
Compared to SABRSD, SOBT does not consider the ordering of the diagonal
entries in the original matrix. 

As another alternative approximation of BRSD, one may also simply
use the upper-triangular part of BRSD as $\hat{\boldsymbol{R}}$ with
the same ordering as SABRSD, instead of minimizing the objective function
\eqref{eq:min-norm}. We refer to this second alternative as \emph{truncated
BRSD} (or \emph{TBRSD}). Compared to SABRSD, TBRSD does not minimize
the errors based on $\epsilon$-accuracy in Definition~\ref{def:epsilon-accuracy};
it is also more expensive in that it requires factoring $\lceil\nicefrac{s}{2}\rceil$
diagonal blocks. In Section~\ref{subsec:Effect-of-singly}, we will
compare SABRSD with SOBT and TBRSD to assess the effectiveness of
our singly-diagonal optimization. 

\subsection{Qualitative comparison of preconditioners}

We compare the advantages and disadvantages of some block preconditioners
qualitatively in terms of near-optimality, the number of real or complex
factorizations, and the number of real or complex solves of the diagonal
blocks. Table~\ref{tab:Comparison-of-preconditioners} summarizes
the comparison of nine block preconditioners, where the first five
were reviewed in Section~\ref{sec:Related-works} and the others
were introduced in this work. Among these preconditioners, BCSD, BRSD,
and BJF are mathematically optimal block preconditioners for GL, assuming
complete factorization and exact arithmetic. Among these three, BJF
suffers from instability for large $s$. PNKP and SABRSD are near-optimal
for large and intermediate time steps. In terms of factorization cost,
we assume a two-way symmetry along its diagonal (i.e., $a_{i,i}=a_{s-i,s-i}$
for $i=1,\dots,\lfloor s\rfloor$) for BD and BGS; we assume that
there are $\lfloor\nicefrac{s}{2}\rfloor$ conjugate pairs of complex
eigenvalues for BC, BJF, and BCSD, so that the factorization cost
can be approximately halved by taking advantage of the fact that $\overline{\boldsymbol{M}^{-1}\boldsymbol{b}}=\overline{\boldsymbol{M}}^{-1}\overline{\boldsymbol{b}}$.
Both of these properties can be verified for GL schemes. Under these
assumptions, assuming near-linear scaling of the approximate factorization,
BRSD requires about twice as much memory and factorization time as
BD and BGS; the costs of BC, BJF, and BCSD are comparable to BRSD.
SABRSD has the lowest memory requirement and factorization cost, comparable
to PNKP and (G)KPS. In terms of the solve time, SABRSD is comparable
with BGS, PNKP, and (G)KPS; BRSD is about twice as expensive; BCSD
requires two to four times floating-point operations due to complex
arithmetic, and similarly for BC and BJF. We note that BD, (G)KPS,
BC, BJF, and PNKP enjoy $s$ independent solves on the diagonal blocks
(aka stage-parallelism \cite{pazner2017stage}), which may be advantageous
for multithreaded implementation on multi-core computers. In contrast,
the others require a block back solve, for which the $s$ solves need
to be performed in a serial order.

\begin{table}
\caption{\label{tab:Comparison-of-preconditioners}Comparison of preconditioners
for $s$-stage GL schemes. }

\centering{}\setlength\tabcolsep{4pt}%
\begin{tabular}{cccc}
\toprule 
preconditioner & near-optimal & factorization & solves\tabularnewline
\midrule
block diagonal \cite{pazner2017stage} & \multirow{2}{*}{$h^{-2}\delta t\rightarrow0$} & \multirow{2}{*}{$\lceil\nicefrac{s}{2}\rceil\times$ $\mathbb{R}^{m\times m}$} & \multirow{2}{*}{$s\times$ $\mathbb{R}^{m\times m}$}\tabularnewline
block Gauss-Seidel \cite{van1997triangularly} &  &  & \tabularnewline
\midrule 
(G)KPS \cite{chen2014splitting,chen2015generalized} & $-$ & $1\times$ $\mathbb{R}^{m\times m}$ & $s\times$ $\mathbb{R}^{m\times m}$\tabularnewline
\midrule 
block circulant \cite{chan2001strang,zhang2011strang} & $-$ & \multirow{3}{*}{$\begin{array}{c}
\lfloor\nicefrac{s}{2}\rfloor\times\mathbb{C}^{m\times m}\,\&\\
(s\text{ mod }2)\times\mathbb{R}^{m\times m}
\end{array}$} & \multirow{3}{*}{$s\times$ $\mathbb{C}^{m\times m}$}\tabularnewline
block Jordan form \cite{huang2019conditioning} & small $s$ &  & \tabularnewline
block CSD (BCSD) & optimal &  & \tabularnewline
\midrule
block RSD (BRSD) & optimal & \multicolumn{2}{c}{$\lfloor\nicefrac{s}{2}\rfloor\times$ $\mathbb{R}^{2m\times2m}$
\& $(s\text{ mod }2)\times$ $\mathbb{R}^{m\times m}$}\tabularnewline
\midrule 
PNKP & $h^{-2}\delta t\rightarrow\infty$ & \multirow{2}{*}{$1\times$ $\mathbb{R}^{m\times m}$} & \multirow{2}{*}{$s\times$ $\mathbb{R}^{m\times m}$}\tabularnewline
SABRSD & $\delta t\gtrsim\mathcal{O}(h)$ &  & \tabularnewline
\bottomrule
\end{tabular}
\end{table}

For BCSD, BRSD, PNKP, and SABRSD, the conclusions in Table~\ref{tab:Comparison-of-preconditioners}
hold for other FIRK schemes, assuming their Butcher matrices have
$\lfloor\nicefrac{s}{2}\rfloor$ distinct complex conjugate pairs
of eigenvalues. However, some conclusions for other preconditioners
may change. In particular, the two-way symmetry of the diagonal entries
does not hold for Radau and some other FIRK schemes, so the number
of factorizations for BD and BGS would be $s$ instead of $\lceil\nicefrac{s}{2}\rceil$.

\section{Numerical Experimentation\label{sec:Numerical-Results}}

In this section, we report numerical experimentation with the preconditioners
introduced in Sections~\ref{sec:Optimal-preconditioners} and \ref{sec:Near-Optimal-Preconditioners},
and compare them with some others that we reviewed in Section~\ref{subsec:State-of-the-Art-Preconditioners}.
We discretized the 3D AD equation \eqref{eq:advection-diffusion}
using both FEM and FDM on $\Omega=[0,1]^{3}$ with $\mu=1$, and we
set velocities to $\boldsymbol{v}=[10,10,10]^{T}$ and $\boldsymbol{v}=[1,1,1]^{T}$
for FEM and FDM, respectively. We conducted our numerical experiments
using the method of \emph{manufactured solutions}. Specifically, we
used the following manufactured solution
\begin{equation}
u\left(x,y,z,t\right)=\sin\left(1.5\pi t\right)\sin\left(\pi x\right)\sin\left(\pi y\right)\sin\left(\pi z\right),\label{eq:3d-manufactured-sol}
\end{equation}
which generalized the 2D test case in \cite[Section 5.2]{chen2014splitting}.
This manufactured solution has homogeneous boundary conditions at
all time steps. For FEM, we used FEniCS 2019.1.0 \cite{alnaes2015fenics,logg2012automated}
to assemble the mass matrix $\boldsymbol{M}_{\text{FEM}}$ and the
stiffness matrix $\boldsymbol{K}_{\text{FEM}}$ in \eqref{eq:mass-striff-FEM}
using a series of tetrahedral meshes with quadratic, cubic, and quartic
(i.e., $P_{2}$, $P_{3}$, and $P_{4}$, correspondingly) Lagrange
elements. Table~\ref{tab:fem-stat} shows the degrees of freedom
(DOF) and the numbers of nonzeros (NNZ) in $\boldsymbol{M}_{\text{FEM}}$
and $\boldsymbol{K}_{\text{FEM}}$ for each testing case. For FDM,
we used our in-house MATLAB code to construct $\boldsymbol{K}_{\text{FDM}}$
in \eqref{eq:mass-stiff-FDM} on a series of equidistant structured
meshes with $h=\nicefrac{1}{32}$, $\nicefrac{1}{64}$, and $\nicefrac{1}{128}$,
where one-sided stencils were used near the boundary. Table~\ref{tab:fdm-stat}
shows the DOFs and NNZ in $\boldsymbol{K}_{\text{FDM}}$ for the testing
cases. In addition, Tables~\ref{tab:fem-stat} and \ref{tab:fdm-stat}
show the cell P{\'e}clet numbers \cite[p. 251]{ern2013theory}, $\text{Pe}_{h}=2h\Vert\boldsymbol{v}\Vert/\mu$,
where we chose $h$ to be the distance between nearest nodes since
we used equidistant Lagrange elements and uniform meshes. It is well
known that when $\text{Pe}_{h}\gg1$ (i.e., when the flow is advection-dominated),
continuous FEM and FDM may suffer from instabilities, so we consider
the cases that are (marginally) diffusion-dominated, with $\text{Pe}_{h}$
spanning two orders of magnitude. For the heat equations, for which
$\text{Pe}_{h}=0$, the behaviors are similar to those with small
$\text{Pe}_{h}$, so we omit them from our comparison.

We discretized the resulting systems of ODEs using the GL schemes.
For an $s$-stage GL scheme, the DOFs of $\boldsymbol{\mathcal{A}}$
in \eqref{eq:coeff_matrix} are $s$ times of those in Tables~\ref{tab:fem-stat}
and~\ref{tab:fdm-stat}. For example, for the five-stage GL scheme
on the finest mesh, there are more than ten million DOFs in $\boldsymbol{\mathcal{A}}$
for both FEM and FDM.

\begin{table}
\caption{\label{tab:fem-stat}Statistics of testing cases using 3D FEM discretization.
DOF and NNZ stand for degrees of freedom and numbers of nonzeros in
$\boldsymbol{K}_{\text{FEM}}$ in \eqref{eq:mass-striff-FEM}, respectively.}

\centering{}\setlength\tabcolsep{2pt}%
\begin{tabular}{ccccccccccccc}
\toprule 
elem. &  & \multicolumn{3}{c}{$\ell=1$} &  & \multicolumn{3}{c}{$\ell=2$} &  & \multicolumn{3}{c}{$\ell=3$}\tabularnewline
\cmidrule{3-5} \cmidrule{4-5} \cmidrule{5-5} \cmidrule{7-9} \cmidrule{8-9} \cmidrule{9-9} \cmidrule{11-13} \cmidrule{12-13} \cmidrule{13-13} 
type &  & DOF & NNZ & $\text{Pe}{}_{h}$ &  & DOF & NNZ & $\text{Pe}{}_{h}$ &  & DOF & NNZ & $\text{Pe}{}_{h}$\tabularnewline
\midrule 
$P_{2}$ &  & 3,375 & 79,183 & 2.17 &  & 29,791 & 776,879 & 1.08 &  & 250,047 & 6,853,231 & 0.54\tabularnewline
$P_{3}$ &  & 12,167 & 497,723 & 1.44 &  & 103,823 & 4,637,699 & 0.72 &  & 857,375 & 39,951,635 & 0.36\tabularnewline
$P_{4}$ &  & 29,791 & 1,936,927 & 1.08 &  & 250,047 & 17,530,591 & 0.54 &  & 2,048,383 & 148,960,351 & 0.27\tabularnewline
\bottomrule
\end{tabular}
\end{table}
\begin{table}
\caption{\label{tab:fdm-stat}Statistics of testing cases using 3D FDM discretization.
DOF and NNZ stand for degrees of freedom and numbers of nonzeros in
$\boldsymbol{K}_{\text{FDM}}$ in \eqref{eq:mass-stiff-FDM}, respectively.}

\centering{}\setlength\tabcolsep{0.75pt}%
\begin{tabular}{ccccccccccccc}
\toprule 
stencil &  & \multicolumn{3}{c}{$h=\nicefrac{1}{32}$} &  & \multicolumn{3}{c}{$h=\nicefrac{1}{64}$} &  & \multicolumn{3}{c}{$h=\nicefrac{1}{128}$}\tabularnewline
\cmidrule{3-5} \cmidrule{4-5} \cmidrule{5-5} \cmidrule{7-9} \cmidrule{8-9} \cmidrule{9-9} \cmidrule{11-13} \cmidrule{12-13} \cmidrule{13-13} 
order &  & DOF & NNZ & $\text{Pe}{}_{h}$ &  & DOF & NNZ & $\text{Pe}{}_{h}$ &  & DOF & NNZ & $\text{Pe}{}_{h}$\tabularnewline
\midrule 
$p=2$ &  & \multirow{3}{*}{29,791} & 202,771 & \multirow{3}{*}{1.08e-01} &  & \multirow{3}{*}{250,047} & 1,726,515 & \multirow{3}{*}{5.41e-02} &  & \multirow{3}{*}{2,048,383} & 14,241,907 & \multirow{3}{*}{2.71e-02}\tabularnewline
$p=4$ &  &  & 381,517 &  &  &  & 3,226,797 &  &  &  & 26,532,205 & \tabularnewline
$p=6$ &  &  & 560,263 &  &  &  & 4,727,079 &  &  &  & 38,822,503 & \tabularnewline
\bottomrule
\end{tabular}
\end{table}

In all of our tests, we solved systems $\boldsymbol{\mathcal{A}}\mathcal{U}=\mathcal{B}$
using our in-house MATLAB implementation of right-preconditioned restarted
GMRES based on \cite{saad1993flexible}. We limited the dimension
of the Krylov subspace to 30 (i.e., GMRES(30)) and limited the maximum
number of iterations to $500$. For the convergence criterion, we
used different relative tolerance (aka \emph{rtol} or $\varepsilon$),
i.e., $\left\Vert \mathcal{B}-\boldsymbol{\mathcal{A}}\mathcal{U}\right\Vert \le\varepsilon\left\Vert \mathcal{B}\right\Vert $.
In terms of the initial guess of GMRES, we used the initial condition
for the first time step and used the solution from the previous time
step starting from second iterations. To verify the correctness of
our implementations of numerical discretization methods and restarted
GMRES, Table~\ref{tab:Verification-of-convergence} shows the relative
2-norm errors compared to the exact solution in \eqref{eq:3d-manufactured-sol}
using $P_{3}$ and $P_{4}$ finite elements as well as fourth and
sixth-order finite differences. For this verification, we chose the
Gauss-Legendre schemes and $\delta t$ to match the order and resolution
of spatial discretization approximately, and we solved the linear
systems using GMRES right-preconditioned using BRSD with HILUCSI.
It can be seen that the errors converged at the expected orders except
for the sixth-order finite differences at the finest level, which
had reached near machine precision. In the remainder of this section,
we will assess the effect of different combinations of spatial and
temporal discretization schemes and resolutions on the preconditioning
techniques for a range of convergence tolerances, with the understanding
that some of them may not be ideal combinations from the point of
view of discretization errors. To get a reliable count on the number
of iterations, we ran each test case for ten time-steps and averaged
the iteration counts starting from the second time step. 
\begin{table}
\caption{\label{tab:Verification-of-convergence}Verification of relative $2$-norm
errors and convergence rates (in parentheses) using FEM with $\boldsymbol{v}=[10,10,10]^{T}$
and FDM with $\boldsymbol{v}=[1,1,1]^{T}$.}

\centering{}\setlength\tabcolsep{4pt}%
\begin{tabular}{cccccc}
\toprule 
\multicolumn{6}{c}{Finite element methods}\tabularnewline
\midrule
$\begin{array}{c}
\text{discretization}\\
\text{order}
\end{array}$ &  &  & $\begin{array}{c}
\ell=1\\
\delta t=\nicefrac{1}{32}
\end{array}$ & $\begin{array}{c}
\ell=2\\
\delta t=\nicefrac{1}{64}
\end{array}$ & $\begin{array}{c}
\ell=3\\
\delta t=\nicefrac{1}{128}
\end{array}$\tabularnewline
\midrule
$P_{3}$/GL-2 &  &  & 3.26e-05 (-) & 2.16e-06 (4.0) & 1.39e-07 (4.0)\tabularnewline
$P_{4}$/GL-3 &  &  & 1.56e-06 (-) & 4.96e-08 (5.1) & 1.56e-09 (5.0)\tabularnewline
\midrule
\midrule 
\multicolumn{6}{c}{Finite difference methods}\tabularnewline
\midrule 
$\begin{array}{c}
\text{discretization}\\
\text{order}
\end{array}$ &  &  & $\begin{array}{c}
h=\nicefrac{1}{32}\\
\delta t=\nicefrac{1}{32}
\end{array}$ & $\begin{array}{c}
h=\nicefrac{1}{64}\\
\delta t=\nicefrac{1}{64}
\end{array}$ & $\begin{array}{c}
h=\nicefrac{1}{128}\\
\delta t=\nicefrac{1}{128}
\end{array}$\tabularnewline
\midrule 
$p=4$/GL-2 &  &  & 5.47e-07 (-) & 3.42e-08 (4.1) & 2.15e-09 (4.0)\tabularnewline
$p=6$/GL-3 &  &  & 2.02e-09 (-) & 3.87e-11 (5.8) & 9.10e-13 (5.5)\tabularnewline
\bottomrule
\end{tabular}
\end{table}

\subsection{Comparison of near-linear approximate inverses\label{subsec:Comparison-of-near-linear}}

For the mathematically optimal preconditioners, including BRSD, BJF,
and BCSD, an iterative solver would converge in a single iteration
to (near) machine precision, which we have proven theoretically and
verified empirically. For practical purpose, we focus on assessing
near-optimal preconditioners, which require near-linear complexity
of approximate inverses in terms of both factorization and solve times
per KSP iteration. As discussed in Section~\ref{subsec:Single-level-and-multilevel},
ILU(0) and HILUCSI are two techniques that satisfy this requirement.
As noted in Section~\ref{subsec:Single-level-and-multilevel}, another
important class of scalable preconditioners are multigrid methods,
especially those with a robust ILU-based smoother. For completeness,
we have tested our preconditioning strategies by leveraging the BoomerAMG
in hypre 2.21.0 \cite{hypre2021}, which we accessed through PETSc
3.15.0 \cite{balay2019petsc}. BoomerAMG has many choices of coarsening,
interpolation, and smoothing options \cite{hypre2021}. After some
preliminary testing, we chose to use its $\ell1$-Gauss-Seidel and
Euclid smoothers as two examples, where the latter is an extensible
incomplete LU framework \cite{hysom2001scalable}. Although hypre
offers other smoothers based on ILU and approximate inverse (in particular,
Pilut and ParaSails), only Euclid enabled hypre to solve all the systems
in our test. We used Falgout coarsening and classical interpolation
options for both smoothers, since our tests show that different coarsening
and interpolation options made little difference for the advection-diffusion
equation.

Table~\ref{tab:ILU-vs-AMG-FEM} compares the average number of GMRES
iterations as well as the average solve times per time step for the
three-stage GL (GL-3) scheme with $P_{3}$ FEM on the two coarser
meshes. As a baseline, we compare ILU(0), HILUCSI, and AMG as ``global''
preconditioners (i.e., approximately factorizing the $sm\times sm$
coefficient matrix $\boldsymbol{\mathcal{A}}$ directly), and then
compare them to factorizing the diagonal blocks in BRSD, KPS, BGS,
and BJF. Although BJF was not previously used as a preconditioner
in the literature, it shares some similarities with the Strang-type
block-circulant preconditioners \cite{chan2001strang}, so we include
it in this baseline comparison.

In terms of the solve times, we include only the (block) triangular
solves and the sparse matrix-vector multiplication, since they are
the overwhelmingly dominant part of the computations, and other parts
of the preconditioned GMRES have virtually the same cost for different
preconditioning techniques. All of timing was conducted on a single
node of a Linux cluster with dual 2.60 GHz Intel(R) Xeon(R) E5-2690v3
processors with 128 GB RAM. HILUCSI is written in C++ \cite{chen2021hilucsi},
and we compiled it using GCC-4.8 with optimization option `-O3 -ffast-math'.
HILUCSI has three key parameters, namely droptol ($\tau$), condest
($\kappa$), and NNZ factor ($\alpha$), for which we used $10^{-3}$,
$5$, and $5$, respectively, for all the tests. For ILU(0), we used
MATLAB's built-in (pre-compiled) \textsf{ilu} function\footnotemark[3]\footnotetext[3]{Unlike
user-defined M-files, MATALB's built-in ilu function is pre-compiled
C code, and there is virtually no performance penalty when calling
it in MATLAB.} with the ``nofill'' option. We chose this option
because the more sophisticated ILUTP leads to superlinear complexity
\cite{ghai2019comparison}. To make the comparison more fair, we applied
preprocessing steps, including equilibration \cite{duff2001algorithms}
and fill-reduction ordering \cite{chan1980linear}, before ILU(0),
to maximize its effectiveness. As can be seen from the table, HILUCSI
and ILU(0)+preprocessing have comparable performance for the coarser
meshes, but HILUCSI outperformed ILU(0) for the finer meshes by up
to a factor of three and five in terms of solve times and numbers
of iterations, respectively, thanks to HILUCSI's multilevel structure,
its attention to the stability of triangular factors, and its scalability-oriented
droppings. Regarding BoomerAMG in hypre, we make the following observations.
For global and BRSD, AMG-GS (i.e., BoomerAMG with $\ell1$-Gauss-Seidel
smoother) failed to enable GMRES to converge after 500 iterations,
so did almost all other combinations we tested. AMG-Euclid (i.e.,
BoomerAMG with Euclid as smoother) managed to converge, with a noticeable
improvement in terms of the number of iterations compared to ILU(0).
However, its runtimes were about 15--20 times slower than ILU(0),
and its numbers of iterations and runtimes were 4--6 and 30--50
times worse than HILUCSI.\footnotemark[4]\footnotetext[4]{BoomerAMG
also required longer setup time than the factorization times of ILU(0)
and HILUCSI.} The poor-performance of AMG for these two cases were
because $P_{3}$ elements in 3D lead to matrices that are much denser
than low-order FEM. The increased density introduces significant challenges
to the coarsening, interpolation, and smoothing strategies in AMG.\footnotemark[5]\footnotetext[5]{Effectiveness
for high-order FEM may be improved by the $p$-multigrid or low-order
multigrid methods; see e.g. \cite{sundar2015comparison}. Hypre has
some support for some edge-based high-order elements \cite{hypre2021},
but it has no support for high-order Lagrange elements. Developing
customized multigrid methods is obviously beyond the scope of this
work.} Furthermore, the $3m\times3m$ and $2m\times2m$ blocks in
GL-3 and in BRSD further increased the number of nonzeros per row
and posed additional challenges. Fortunately, by taking advantage
of Kronecker-product structures, BoomerAMG significantly improved
its robustness, numbers of iterations, and runtimes. Although AMG-Euclid
still had the worst performance, it was only a factor of eight slower
than HILUCSI, which can be overcome more easily by utilizing better
algorithms (such as $p$-multigrid \cite{sundar2015comparison} or
hybrid geometric+algebraic multigrid \cite{lu2014hybrid}) or by leveraging
parallel computations. Hence, there was a significant advantage to
leverage block preconditioners for AMG, rather than applying AMG on
the global coefficient matrix. Regarding BJF, for HILUCSI and ILU(0),
it preconditioners required a similar number of iterations as the
global preconditioner, with some advantages in terms of runtimes.
More importantly, BJF outperformed KPS and BGS by more than a factor
of two on average in terms of numbers of iterations when HILUCSI was
used. Unfortunately, we could not test hypre with BJF since it does
not yet support complex arithmetic \cite{hypre2021}. Between BGS
and KPS, they have better performance for smaller and larger $h^{-2}\delta t$,
respectively, confirming our asymptotic analysis in Section~\ref{subsec:State-of-the-Art-Preconditioners}.
We observed similar behavior for FDM. 

\begin{table}
\caption{Comparison of average numbers of GMRES iterations and solve times
per time step using HILUCSI, ILU(0) with preprocessing, and BoomerAMG
with $\ell1$-Gauss-Seidel and Euclid smoothers in hypre as underlying
preconditioners, for GL-3 with $P_{3}$ FEM for the AD equation with
$\boldsymbol{v}=[10,10,10]^{T}$ and $\text{rtol}=10^{-8}$. Times
are in seconds. `$-$' indicates non-convergence after 500 GMRES iterations.
Leaders within each group are in boldface.\label{tab:ILU-vs-AMG-FEM}}

\centering{}\setlength\tabcolsep{2pt}%
\begin{tabular}{ccccccccccccccccc}
\toprule 
\multirow{3}{*}{method} & \multirow{3}{*}{ILU} &  & \multicolumn{4}{c}{$\delta t=\nicefrac{1}{4}$} &  & \multicolumn{4}{c}{$\delta t=\nicefrac{1}{16}$} &  & \multicolumn{4}{c}{$\delta t=\nicefrac{1}{64}$}\tabularnewline
\cmidrule{4-7} \cmidrule{5-7} \cmidrule{6-7} \cmidrule{7-7} \cmidrule{9-12} \cmidrule{10-12} \cmidrule{11-12} \cmidrule{12-12} \cmidrule{14-17} \cmidrule{15-17} \cmidrule{16-17} \cmidrule{17-17} 
 &  &  & \multicolumn{2}{c}{$\ell=1$} & \multicolumn{2}{c}{$\ell=2$} &  & \multicolumn{2}{c}{$\ell=1$} & \multicolumn{2}{c}{$\ell=2$} &  & \multicolumn{2}{c}{$\ell=1$} & \multicolumn{2}{c}{$\ell=2$}\tabularnewline
\cmidrule{4-7} \cmidrule{5-7} \cmidrule{6-7} \cmidrule{7-7} \cmidrule{9-12} \cmidrule{10-12} \cmidrule{11-12} \cmidrule{12-12} \cmidrule{14-17} \cmidrule{15-17} \cmidrule{16-17} \cmidrule{17-17} 
 &  &  & iter. & time & iter. & time &  & iter. & time & iter. & time &  & iter. & time & iter. & time\tabularnewline
\midrule
\multirow{4}{*}{global} & HILUCSI &  & \textbf{5.4} & \textbf{0.41} & \textbf{9.2} & \textbf{8.13} &  & \textbf{4.9} & \textbf{0.38} & \textbf{8.3} & \textbf{7.33} &  & \textbf{3.2} & \textbf{0.22} & \textbf{5.8} & \textbf{4.49}\tabularnewline
 & ILU(0)+pre &  & 27.6 & 0.51 & 66.8 & 11.1 &  & 24.8 & 0.4 & 57.7 & 9.29 &  & 16.1 & 0.34 & 35.2 & 5.79\tabularnewline
 & AMG-GS &  & $-$ & $-$ & $-$ & $-$ &  & $-$ & $-$ & $-$ & $-$ &  & $-$ & $-$ & $-$ & $-$\tabularnewline
 & AMG-Euclid &  & 20.3 & 9.0 & 47.0 & 281.5 &  & 17.3 & 7.27 & 38 & 229 &  & 11.1 & 4.70 & 21.8 & 135\tabularnewline
\midrule 
\multirow{4}{*}{BRSD} & HILUCSI &  & \textbf{5.2} & \textbf{0.22} & \textbf{9.1} & \textbf{3.27} &  & \textbf{4.9} & \textbf{0.19} & \textbf{8.0} & \textbf{3.07} &  & \textbf{3.1} & \textbf{0.11} & \textbf{5.0} & \textbf{1.86}\tabularnewline
 & ILU(0)+pre &  & 29.3 & 0.4 & 81.6 & 9.61 &  & 25.1 & 0.34 & 64.7 & 7.95 &  & 15.8 & 0.22 & 34.1 & 4.17\tabularnewline
 & AMG-GS &  & $-$ & $-$ & $-$ & $-$ &  & $-$ & $-$ & $-$ & $-$ &  & $-$ & $-$ & $-$ & $-$\tabularnewline
 & AMG-Euclid &  & 23.7 & 6.0 & 53.3 & 129 &  & 20.6 & 4.35 & 52.8 & 128 &  & 12.2 & 2.69 & 31.4 & 67.0\tabularnewline
\midrule 
\multirow{4}{*}{KPS} & HILUCSI &  & \textbf{12.8} & \textbf{0.2} & \textbf{14.6} & \textbf{3.23} &  & \textbf{12.3} & 0.3 & \textbf{13.1} & \textbf{3.25} &  & \textbf{10.1} & 0.22 & \textbf{10.4} & \textbf{2.4}\tabularnewline
 & ILU(0)+pre &  & 51.8 & 0.4 & 125 & 10.1 &  & 33.9 & \textbf{0.28} & 74.8 & 6.73 &  & 18.3 & \textbf{0.17} & 35.8 & 3.54\tabularnewline
 & AMG-GS &  & 15.8 & 1.58 & 17.9 & 19.9 &  & 14.0 & 1.48 & 15.7 & 17.0 &  & 10.3 & 0.97 & 11.1 & 12.7\tabularnewline
 & AMG-Euclid &  & 15.8 & 2.23 & 17.7 & 26.7 &  & 14.1 & 2.07 & 15.4 & 23.5 &  & 11.0 & 1.56 & 12.0 & 18.7\tabularnewline
\midrule 
\multirow{4}{*}{BGS} & HILUCSI &  & \textbf{9.8} & \textbf{0.25} & \textbf{13.1} & \textbf{3.32} &  & \textbf{9.4} & \textbf{0.23} & \textbf{11.8} & \textbf{3.01} &  & \textbf{7.2} & \textbf{0.17} & \textbf{8.0} & \textbf{1.92}\tabularnewline
 & ILU(0)+pre &  & 83.4 & 0.74 & 349 & 32.7 &  & 52.7 & 0.49 & 166 & 15.2 &  & 21.3 & 0.22 & 52.0 & 4.79\tabularnewline
 & AMG-GS &  & 15.4 & 1.61 & 18.8 & 19.6 &  & 13.3 & 1.35 & 15.9 & 17.9 &  & 8.8 & 0.88 & 10.6 & 10.8\tabularnewline
 & AMG-Euclid &  & 15.1 & 2.07 & 18.4 & 27.9 &  & 13.6 & 2.02 & 15.4 & 25.4 &  & 9.8 & 1.40 & 11.2 & 16.6\tabularnewline
\midrule
\multirow{2}{*}{BJF} & HILUCSI &  & \textbf{4.7} & \textbf{0.29} & \textbf{7.8} & \textbf{2.68} &  & \textbf{4.1} & \textbf{0.21} & \textbf{5.3} & \textbf{2.45} &  & \textbf{3.2} & \textbf{0.14} & \textbf{4.9} & \textbf{1.68}\tabularnewline
 & ILU(0)+pre &  & 27.6 & 0.45 & 66.8 & 8.98 &  & 24.8 & 0.4 & 57.7 & 7.85 &  & 16.1 & 0.29 & 35.2 & 5.07\tabularnewline
\bottomrule
\end{tabular}
\end{table}

Based on these previous comparisons, we will use HILUCSI in our later
comparative studies. Note that our previous timing results in Table~\ref{tab:ILU-vs-AMG-FEM}
did not consider factorization cost. It is important to verify the
near-linear complexity of its factorization and solve times. To this
end, we solved the 3D AD equation under mesh refinement using FEM
and FDM with the two-, three-, and four-stage GL schemes with $\delta t=1/4$
but different orders of spatial accuracy. For simplicity, we only
considered the PNKP in \eqref{eq:PNKP} to precondition GMRES. Figure~\ref{fig:SAIK-time}
reports the factorization times and the averaged solve time per GMRES
iteration at the first time step. These scalability tests were conducted
on a Linux cluster with 2.5 GHz Intel(R) Xeon(R) E5-2680v3 processors
and 64 GB of RAM. It is clear that HILUCSI scaled (near) linearly
in both factorization and solve times, confirming the analysis in
\cite{chen2021hilucsi}. Note that for $P_{2}$ FEM, the super-linear
time complexity was due to the cache performance of the dense solve
at the coarsest level in HILUCSI. It is also worth noting that the
factorization cost is about an order of magnitude of that of solve
times. The factorization is performed only at the first time step
or when $\delta t$ changes, so its cost can be amortized over several
time steps. In our later analysis, we will focus on comparing the
number of GMRES iterations. As a rule of thumb, the cost per GMRES
iteration in BCSD, BJF, and BRSD is approximately twice as much as
PNKP, and the cost of the others is comparable to PNKP.

\begin{figure}
\subfloat[\label{fig:SAIK-time-fem}Runtimes for 3D FEM.]{\includegraphics[width=0.48\columnwidth]{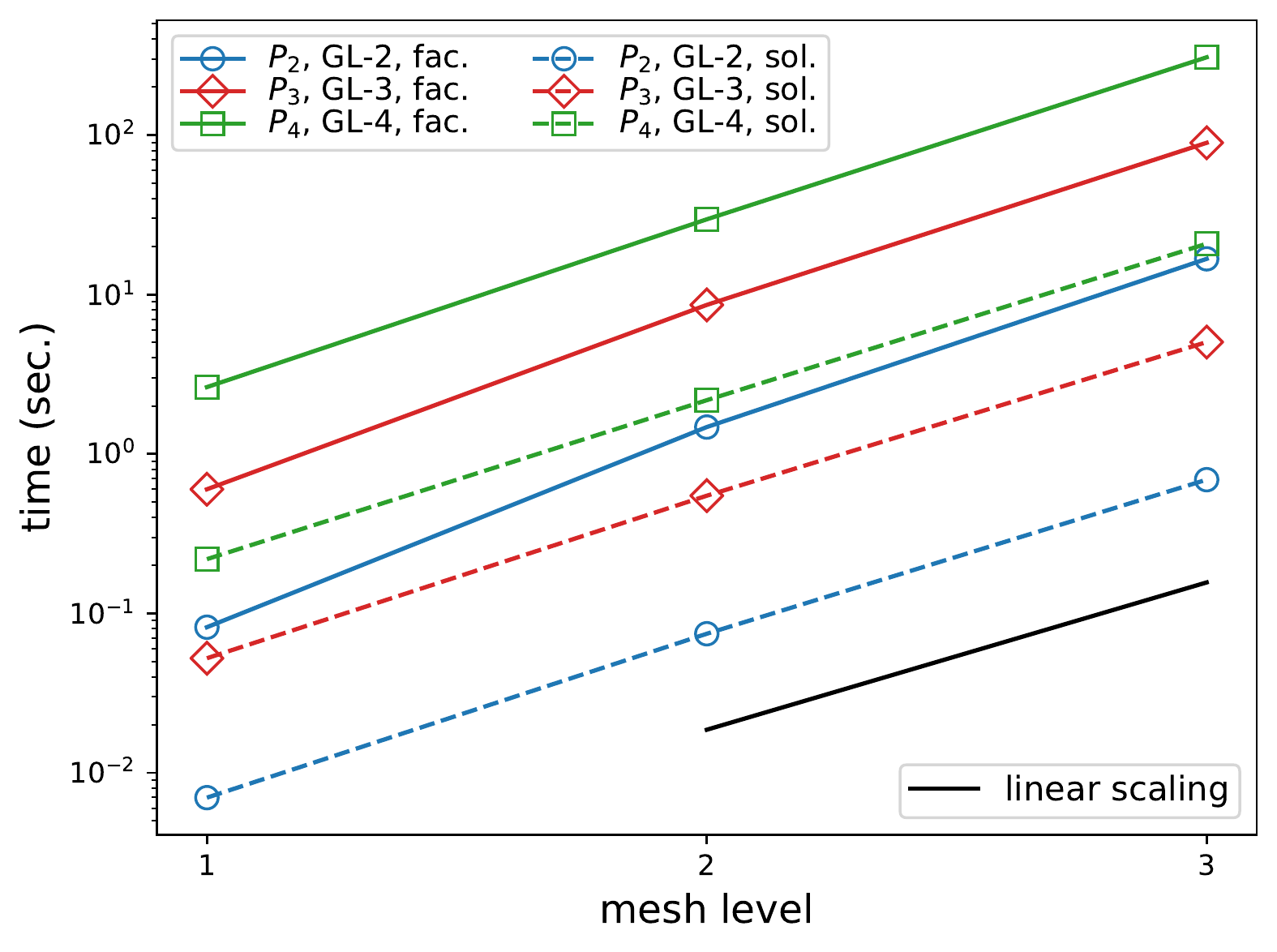}

}\hfill\subfloat[\label{fig:SAIK-time-fdm}Runtimes for 3D FDM.]{\includegraphics[width=0.48\columnwidth]{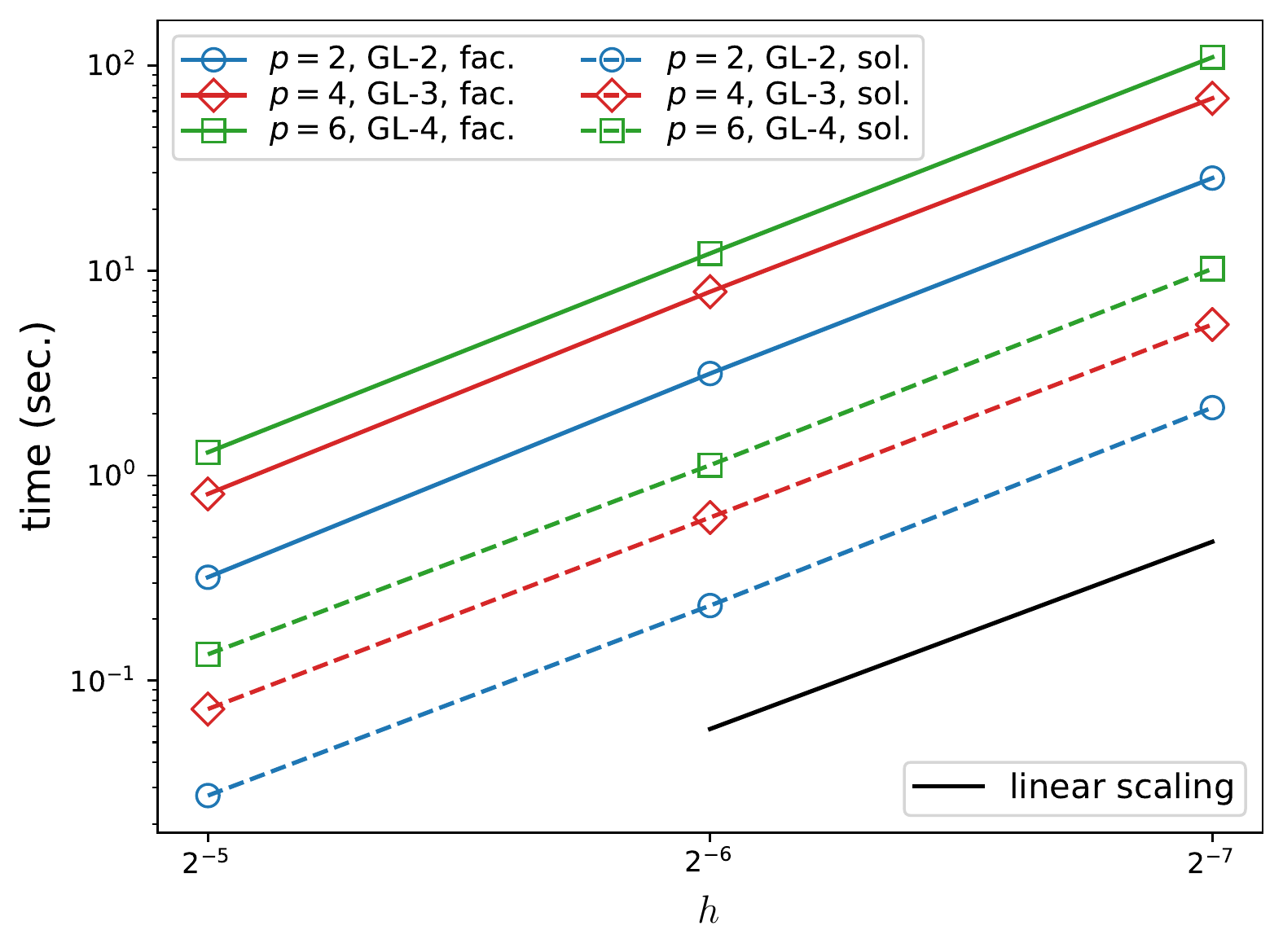}

}

\caption{\label{fig:SAIK-time}Factorization and solve times of PNKP for 3D
FEM and FDM. Solid and dashed lines are for the factorization times
of HILUCSI and averaged solve times per GMRES iteration, respectively.}
\end{figure}

\subsection{Comparison of mathematically optimal preconditioners\label{subsec:Results-optimal-preconditioners}}

We compare the mathematically optimal preconditioners, namely BCSD,
BRSD, and BJF. As shown in Tables~\ref{tab:ILU-vs-AMG-FEM}, BJF
significantly outperformed KPS and BGS in terms of the number of GMRES
iterations. Hence, it suffices to compare BCSD and BRSD with BJF.
In addition, since PNKP is optimal as $h^{-2}\delta t\rightarrow\infty$,
we also include it as a point of reference due to its simplicity.

As in Section~\ref{subsec:Comparison-of-near-linear}, we discretized
the 3D AD equation \eqref{eq:advection-diffusion} using both FEM
and FDM with different velocities. For FEM, we used the two- and three-stage
GL schemes (i.e., GL-2 and GL-3, respectively), which converge at
the rates of $\mathcal{O}\left(\delta t^{4}\right)$ and $\mathcal{O}\left(\delta t^{6}\right)$
in time, correspondingly. For GL-2 and GL-3, we used $P_{2}$ and
$P_{3}$ elements for spatial discretization, correspondingly, where
$P_{k}$ FEM converges at $\mathcal{O}\left(h^{k+1}\right)$ in $L^{2}$
norm. We used $\text{rtol}=10^{-6}$ and $10^{-10}$ as the convergence
tolerance in GMRES for GL-2 and GL-3, respectively. Table~\ref{tab:fem-external-comp}
shows the average numbers of GMRES iterations. It is worth noting
that BCSD and BJF had comparable performance, and they both significantly
outperformed PNKP on coarser meshes, although PNKP is quite competitive
with larger time steps on the finest mesh. BRSD performed slightly
worse than BCSD and BJF because its larger $2m\times2m$ diagonal
blocks tend to introduce more droppings.

\begin{table}
\caption{Comparison of the average numbers of GMRES iterations preconditioned
by BCSD, BRSD, BRSD, and PNKP for GL-2 and GL-3 with $P_{2}$ and
$P_{3}$ FEM for the AD equation with $\boldsymbol{v}=[10,10,10]^{T}$
and $\text{rtol}=10^{-6}$ and $10^{-8}$, respectively. Numbers in
parentheses indicate numbers of stages. The leaders are in boldface.\label{tab:fem-external-comp}}

\centering{}\setlength\tabcolsep{5pt}%
\begin{tabular}{ccccccccccccc}
\toprule 
\multirow{2}{*}{method} &  & \multicolumn{3}{c}{$\delta t=\nicefrac{1}{4}$} &  & \multicolumn{3}{c}{$\delta t=\nicefrac{1}{16}$} &  & \multicolumn{3}{c}{$\delta t=\nicefrac{1}{64}$}\tabularnewline
\cmidrule{3-5} \cmidrule{4-5} \cmidrule{5-5} \cmidrule{7-9} \cmidrule{8-9} \cmidrule{9-9} \cmidrule{11-13} \cmidrule{12-13} \cmidrule{13-13} 
 &  & $\ell=1$ & $2$ & $3$ &  & $\ell=1$ & $2$ & $3$ &  & $\ell=1$ & $2$ & $3$\tabularnewline
\midrule
BCSD(2) &  & \textbf{4.0} & 6.0 & 10.3 &  & 4.0 & \textbf{5.3} & 9.2 &  & 3.0 & 4.0 & \textbf{6.0}\tabularnewline
BJF(2) &  & 4.2 & \textbf{5.9} & \textbf{10.1} &  & 4.0 & \textbf{5.3} & \textbf{9.1} &  & 3.0 & 4.0 & 6.8\tabularnewline
BRSD(2) &  & \textbf{4.0} & 6.4 & 11.9 &  & \textbf{3.3} & 6.0 & 10.8 &  & \textbf{2.4} & \textbf{3.9} & 7.1\tabularnewline
PNKP(2) &  & 6.2 & 6.4 & 11.6 &  & 9.0 & 8.7 & 10.4 &  & 12.0 & 12.2 & 9.4\tabularnewline
\midrule 
BCSD(3) &  & 4.8 & 7.9 & 16.3 &  & \textbf{4.1} & \textbf{6.7} & \textbf{12.8} &  & \textbf{3} & \textbf{4.6} & \textbf{7.3}\tabularnewline
BJF(3) &  & \textbf{4.7} & \textbf{7.8} & \textbf{14.7} &  & 4.1 & 7.1 & 13.1 &  & 3.2 & 4.9 & 8.8\tabularnewline
BRSD(3) &  & 5.2 & 9.1 & 17.7 &  & 4.9 & 8.0 & 15.0 &  & 3.1 & 5.0 & 9.8\tabularnewline
PNKP(3) &  & 9.3 & 8.9 & 15.0 &  & 14.6 & 13.8 & 14.6 &  & 24.0 & 23.1 & 18.3\tabularnewline
\bottomrule
\end{tabular}
\end{table}

For FDM, we used second- and sixth-order finite differences (i.e.,
$p=2$ and $p=6$, respectively) in space, and used the GL-2 and GL-4
in time, correspondingly. Table~\ref{tab:fdm-external-comp} shows
the average numbers of GMRES iterations. Similar to the FEM results,
BCSD and BJF are mostly comparable, but BCSD outperformed BJF on average.
BRSD performed slightly worse than BCSD and BJF. PNKP performed the
worse overall, by up to a factor of eight in one case. Hence, the
more sophisticated manipulations in BCSD, BJF, and BCSD are beneficial
in general, compared to the simple PNKP.

\begin{table}
\caption{Comparison of the average numbers of GMRES iterations preconditioned
by BCSD, BRSD, BRSD, and PNKP for GL-2 and GL-4 with FDM for the AD
equation with $\boldsymbol{v}=[1,1,1]^{T}$ and $\text{rtol}=10^{-6}$
and $10^{-10}$, respectively. Numbers in parentheses indicate numbers
of stages. The leaders in each group are in boldface.\label{tab:fdm-external-comp}}

\centering{}\setlength\tabcolsep{4pt}%
\begin{tabular}{ccccccccccccc}
\toprule 
\multirow{2}{*}{method} &  & \multicolumn{3}{c}{$\delta t=\nicefrac{1}{4}$} &  & \multicolumn{3}{c}{$\delta t=\nicefrac{1}{16}$} &  & \multicolumn{3}{c}{$\delta t=\nicefrac{1}{64}$}\tabularnewline
\cmidrule{3-5} \cmidrule{4-5} \cmidrule{5-5} \cmidrule{7-9} \cmidrule{8-9} \cmidrule{9-9} \cmidrule{11-13} \cmidrule{12-13} \cmidrule{13-13} 
 &  & $h=\nicefrac{1}{32}$ & $\nicefrac{1}{64}$ & $\nicefrac{1}{128}$ &  & $h=\nicefrac{1}{32}$ & $\nicefrac{1}{64}$ & $\nicefrac{1}{128}$ &  & $h=\nicefrac{1}{32}$ & $\nicefrac{1}{64}$ & $\nicefrac{1}{128}$\tabularnewline
\midrule
BCSD(2) &  & 12.0 & 23.9 & 55.1 &  & \textbf{9.2} & \textbf{17.1} & \textbf{33.9} &  & \textbf{4.8} & \textbf{8.4} & \textbf{15.9}\tabularnewline
BJF(2) &  & 12.1 & 23.2 & 55.1 &  & 9.7 & 18.1 & 37.2 &  & 5.6 & 9.7 & 18.6\tabularnewline
BRSD(2) &  & \textbf{11.8} & \textbf{22.2} & \textbf{49.8} &  & 9.4 & 17.4 & 34.0 &  & 5.6 & 9.7 & 18.2\tabularnewline
PNKP(2) &  & 12.0 & 23.2 & 55.1 &  & 9.9 & 18.7 & 37.3 &  & 9.0 & 10.4 & 18.9\tabularnewline
\midrule 
BCSD(4) &  & 13.6 & 27.2 & 62.8 &  & \textbf{9.1} & \textbf{17.3} & \textbf{35.9} &  & \textbf{5.9} & \textbf{9.3} & \textbf{18.4}\tabularnewline
BJF(4) &  & 13.6 & 25.8 & 59.9 &  & 10.7 & 20.3 & 43.9 &  & 6.7 & 11.4 & 22.6\tabularnewline
BRSD(4) &  & \textbf{13.0} & \textbf{25.1} & \textbf{58.1} &  & 10.3 & 19.4 & 40.3 &  & 6.6 & 11.2 & 21.7\tabularnewline
PNKP(4) &  & 15.8 & 26.8 & 61.3 &  & 24.4 & 23.6 & 47.0 &  & 49.4 & 28.7 & 26.9\tabularnewline
\bottomrule
\end{tabular}
\end{table}

\subsection{Effect of ordering in CSD and RSD\label{subsec:Effect-of-ordering}}

In this work, we suggest to sort the diagonal blocks in BCSD (and
SABRSD) based on the ascending order of $\left|d_{i}\right|$, due
to Proposition~\ref{prop:error-bnd} and an assumption of insensitivity
of $\left\Vert \delta\boldsymbol{D}_{j}\boldsymbol{x}_{j}\right\Vert $
and $\left\Vert \boldsymbol{R}_{ij}\right\Vert $ to ordering. This
analysis and assumption require some numerical verification to demonstrate
the potential benefits. Table~\ref{tab:fdm-ordering} compares the
average numbers of GMRES iterations for the GL-5 and GL-6 with sixth-order
finite difference method for the 3D AD equation. We ordered the diagonal
entries in CSD in ascending and descending orders of $\left|d_{i}\right|$
for BCSD and BCSD-R, respectively; similarly for SABRSD and SABRSD-R.
It can be seen that BCSD and SABRSD outperformed BCSD-R and SABRSD-R,
respectively. The benefit for SABRSD is particularly significant,
with a margin of about $30\%$ on average. Furthermore, SABRSD-R for
GL-5 did not reach the $10^{-12}$ relative tolerance and stagnated
at about $10^{-11}$. We note that for GL-3 and GL-4, the ordering
also led to a difference with similar margins for SABRSD, but its
effect on BCSD was not as pronounced.

\begin{table}
\caption{Assessment of the impact of ordering on BCSD and SABRSD on the average
numbers of GMRES iterations for the GL-5 and GL-6 with sixth-order
FDM for the AD equation with $\boldsymbol{v}=[1,1,1]^{T}$ and $\text{rtol}=10^{-12}$.
The numbers in the parentheses indicate numbers of stages. `$-$'
indicates stagnation. The leaders in each group are in boldface.\label{tab:fdm-ordering}}

\centering{}\setlength\tabcolsep{3pt}%
\begin{tabular}{ccccccccccccc}
\toprule 
\multirow{2}{*}{method} &  & \multicolumn{3}{c}{$\delta t=\nicefrac{1}{4}$} &  & \multicolumn{3}{c}{$\delta t=\nicefrac{1}{16}$} &  & \multicolumn{3}{c}{$\delta t=\nicefrac{1}{64}$}\tabularnewline
\cmidrule{3-5} \cmidrule{4-5} \cmidrule{5-5} \cmidrule{7-9} \cmidrule{8-9} \cmidrule{9-9} \cmidrule{11-13} \cmidrule{12-13} \cmidrule{13-13} 
 &  & $h=\nicefrac{1}{32}$ & $\nicefrac{1}{64}$ & $\nicefrac{1}{128}$ &  & $h=\nicefrac{1}{32}$ & $\nicefrac{1}{64}$ & $\nicefrac{1}{128}$ &  & $h=\nicefrac{1}{32}$ & $\nicefrac{1}{64}$ & $\nicefrac{1}{128}$\tabularnewline
\midrule
BCSD(5) &  & \textbf{15.8} & 32.3 & \textbf{77.2} &  & \textbf{11.0} & \textbf{20.6} & \textbf{44.8} &  & 7.0 & \textbf{10.9} & \textbf{22.2}\tabularnewline
BCSD-R(5) &  & 15.8 & \textbf{30.9} & 82.9 &  & 11.2 & 22.3 & 49.4 &  & \textbf{6.9} & 11.7 & 24.0\tabularnewline
\midrule
SABRSD(5) &  & \textbf{23.7} & \textbf{56.9} & \textbf{171.6} &  & \textbf{18.9} & \textbf{26.0} & \textbf{63.9} &  & \textbf{17.1} & \textbf{17.4} & \textbf{23.4}\tabularnewline
SABRSD-R(5) &  & 33.8 & 75.7 & $-$ &  & 28.1 & 36.4 & 79.8 &  & 25.1 & 25.4 & 32.2\tabularnewline
\midrule
BCSD(6) &  & \textbf{15.7} & \textbf{32.8} & \textbf{80.0} &  & \textbf{11.0} & \textbf{20.6} & \textbf{46.9} &  & 6.7 & \textbf{10.9} & \textbf{21.6}\tabularnewline
BCSD-R(6) &  & 15.8 & 33.0 & 84.8 &  & \textbf{11.0} & 22.0 & 50.8 &  & \textbf{6.1} & 11.3 & 23.0\tabularnewline
\midrule 
SABRSD(6) &  & \textbf{25.7} & \textbf{58.0} & \textbf{168.8} &  & \textbf{22.7} & \textbf{26.1} & \textbf{61.8} &  & \textbf{21.4} & \textbf{20.4} & \textbf{23.4}\tabularnewline
SABRSD-R(6) &  & 35.9 & 76.4 & 209.4 &  & 32.8 & 36.9 & 77.8 &  & 29.4 & 29.3 & 33.2\tabularnewline
\bottomrule
\end{tabular}
\end{table}

As a side product, Table~\ref{tab:fdm-ordering} also shows that
the number of GMRES iterations of BCSD is about 60\% of that of SABRSD
on average. Since BCSD (and similar BRSD) is approximately twice as
expensive as SABRSD per GMRES iteration, SABRSD is competitive as
a limited-memory, near-optimal preconditioner.

\subsection{Effect of optimization in SABRSD\label{subsec:Effect-of-singly}}

We now assess the effectiveness of the optimization techniques of
SABRSD in Section~\ref{subsec:BARSD}. To this end, we compare SABRSD
with TBRSD, which simply truncates the lower-triangular part of the
sorted and permuted $\boldsymbol{R}$ matrix in RSD. Table~\ref{tab:fem-optimization}
shows the average numbers of GMRES iterations for GL-3 and GL-4 with
$P_{3}$ and $P_{4}$ finite element methods, respectively, for the
3D AD equation. It can be seen that SABRSD always outperformed TBRSD,
and it reduced the number of GMRES iterations by about 20\% on average.
As a point of reference, we also compared SABRSD with SOBT, which
applies the optimization strategy on the transpose of the Butcher
array. Our results show that SOBT improved over BGS up to about 10\%
on finer meshes for GL-3, but it did not improve for GL-4. These results
indicate that our asymptotic analysis in Section~\ref{subsec:BARSD}
works well only when $h^{-2}\delta t$ is sufficiently large, and
the optimization works the best when combined with the ordering described
in Section~\ref{subsec:Ordering-and-permuting}. As a cross-reference
with existing methods, Table~\ref{tab:fem-optimization} also included
the number of GMRES iterations of KPS \cite{chen2014splitting}. It
can be seen that SABRSD outperformed BGS and KPS by 17\% and 32\%
on average, respectively, although it slightly under-performed KPS
in two cases with the largest $\delta t$ on the finest mesh.

\begin{table}
\caption{Comparison of the effect of singly-diagonal optimization for the average
numbers of GMRES iterations for GL-3 and GL-4 with $P_{3}$ and $P_{4}$
FEM for AD equation with $\boldsymbol{v}=[10,10,10]^{T}$ and $\text{rtol}=10^{-2s-2}$,
respectively. The numbers in the parentheses indicate numbers of stages.
The leaders in each group are in boldface.\label{tab:fem-optimization}}

\centering{}\setlength\tabcolsep{5pt}%
\begin{tabular}{ccccccccccccc}
\toprule 
\multirow{2}{*}{method} &  & \multicolumn{3}{c}{$\delta t=\nicefrac{1}{4}$} &  & \multicolumn{3}{c}{$\delta t=\nicefrac{1}{16}$} &  & \multicolumn{3}{c}{$\delta t=\nicefrac{1}{64}$}\tabularnewline
\cmidrule{3-5} \cmidrule{4-5} \cmidrule{5-5} \cmidrule{7-9} \cmidrule{8-9} \cmidrule{9-9} \cmidrule{11-13} \cmidrule{12-13} \cmidrule{13-13} 
 &  & $\ell=1$ & $2$ & $3$ &  & $\ell=1$ & $2$ & $3$ &  & $\ell=1$ & $2$ & $3$\tabularnewline
\midrule
SABRSD(3) &  & \textbf{8.4} & \textbf{11.4} & 22.8 &  & \textbf{8.3} & \textbf{8.8} & \textbf{16.4} &  & \textbf{6.9} & \textbf{6.9} & \textbf{7.8}\tabularnewline
TBRSD(3) &  & 9.8 & 13.4 & 24.8 &  & 10.7 & 12.6 & 19.8 &  & 8.0 & 8.7 & 11.3\tabularnewline
SOBT(3) &  & 11.0 & 13.6 & 25.0 &  & 10.7 & 11.3 & 18.4 &  & 8.6 & 8.8 & 9.8\tabularnewline
BGS(3) &  & 9.8 & 13.1 & 25.4 &  & 9.4 & 11.8 & 19.6 &  & 7.2 & 8.0 & 11.2\tabularnewline
KPS(3) &  & 12.8 & 14.6 & \textbf{22.7} &  & 12.3 & 13.1 & 17.7 &  & 10.1 & 10.4 & 12.2\tabularnewline
\midrule
SABRSD(4) &  & \textbf{13.2} & \textbf{15.6} & 32.9 &  & \textbf{13.3} & \textbf{13.3} & \textbf{20.6} &  & \textbf{11.3} & \textbf{11.7} & \textbf{11.7}\tabularnewline
TBRSD(4) &  & 16.0 & 20.6 & 41.3 &  & 16.7 & 19.2 & 28.3 &  & 13.3 & 13.9 & 17.0\tabularnewline
SOBT(4) &  & 19.6 & 22.6 & 45.0 &  & 19.9 & 20.7 & 27.3 &  & 16.0 & 16.4 & 17.1\tabularnewline
BGS(4) &  & 15.4 & 19.1 & 43.4 &  & 15.1 & 17.6 & 28.2 &  & 11.6 & 12.6 & 16.9\tabularnewline
KPS(4) &  & 18.1 & 20.0 & \textbf{30.4} &  & 17.3 & 17.8 & 22.6 &  & 14.7 & 15.4 & 17.3\tabularnewline
\bottomrule
\end{tabular}
\end{table}

\section{\label{sec:conclusions}Conclusion}

We introduced three mathematically optimal preconditioners (namely
\emph{BRCD}, \emph{BRSD}, and \emph{BJF}) and a limited-memory near-optimal
preconditioner\emph{ SABRSD}, for fully implicit RK schemes in solving
parabolic PDEs. The optimal preconditioners have high memory requirements
and factorization costs. In comparison, SABRSD has a comparable memory
requirement and factorization cost as a singly diagonally implicit
RK scheme (SDIRK). We optimized SABRSD based on the mathematical theory
of $\epsilon$-accurate preconditioners. We approximately factorized
the diagonal blocks in these preconditioners using near-linear complexity
ILU factorizations, including HILUCSI and ILU(0). With HILUCSI, we
then showed that BCSD and BRSD significantly outperformed the prior
state-of-the-art block preconditioners in terms of the number of GMRES
iterations. SABRSD outperformed the prior state-of-the-art block preconditioners
(namely, BGS and KPS) by 17--32\% on average. SABRSD also compared
reasonably well with BCSD in terms of the computational cost while
requiring much less memory. 

In this work, we focused on serial computations using HILUCSI in the
majority of the numerical experimentation. For larger-scale problems
on massively parallel computers, however, we expect leveraging multigrid
methods (especially geometric multigrid, $p$-multigrid, or a hybrid
of them with AMG) with a robust, parallel ILU-based smoother can enjoy
similar benefits that we have demonstrated in serial. In addition,
the potential benefits of BCSD and BJF may be realized for multigrid
solvers that support complex arithmetic.

This work primarily focused on the Gauss-Legendre schemes for their
optimal accuracy. However, our techniques can be generalized to other
A-stable FIRK (such as Radau IIA and Lobatto IIIC) and DIRK schemes
to develop low-memory, near-optimal block preconditioners. While SABRSD
minimizes memory usage compared to BCSD and BRSD, it is worth developing
flexible variants with different diagonal entries, which require generalizing
the error analysis and the optimization strategies developed in this
work. Finally, this work focused on linear PDEs with time-invariant
coefficients. For nonlinear PDEs with time-dependent coefficients,
the discretization may lose the exact Kronecker-product structure.
Another future research direction is to extend the work to support
nonlinear PDEs and compare different linearization strategies in constructing
Kronecker-product approximations.

\section*{Acknowledgments}

The authors thank the support of the CANGA project under the Scientific
Discovery through Advanced Computing (SciDAC) program in the US Department
of Energy\textquoteright s Office of Science, Office of Advanced Scientific
Computing Research through subcontract \#462974 with Los Alamos National
Laboratory. Computational results were obtained using the Seawulf
cluster at the Institute for Advanced Computational Science of Stony
Brook University, which was partially funded by the Empire State Development
grant NYS \#28451. We thank the anonymous reviewers for their helpful
comments, which have significantly improved the presentation of the
work.

\bibliographystyle{siamplain}
\bibliography{reference_rk}

\appendix

\section{Sorting and Permuting RSD\label{sec:sort-permute-RSD}}

Real Schur decomposition is not unique, and BRSD and SABRSD use different
orderings and permutations of the diagonal blocks.  To sort RSD,
we use the \textsf{SRSchur} function in \cite{brandts2002matlab},
except that we need to replace its \textsf{select} sub-function to
sort by the real part of the diagonal blocks. In terms of the orientation,
let $\boldsymbol{R}_{i}$ denote a 2-by-2 diagonal block in $\boldsymbol{R}_{*}$
from the sorted RSD $\boldsymbol{A}=\boldsymbol{Q}_{*}\boldsymbol{R}_{*}\boldsymbol{Q}_{*}^{T}$,
and let $\boldsymbol{S}_{i}$ be its corresponding $2\times2$ permutation
matrix. Let $\boldsymbol{P}_{i}$ denote the $s\times s$ permutation
matrix $\boldsymbol{P}_{i}=\begin{bmatrix}\boldsymbol{I}_{s_{1}}\\
 & \boldsymbol{S}_{i}\\
 &  & \boldsymbol{I}_{s_{2}}
\end{bmatrix}$, where $s_{1}$ and $s_{2}$ are the numbers of rows above and below
$\boldsymbol{R}_{i}$, respectively. The sorted and permuted RSD
is then
\begin{equation}
\boldsymbol{A}=\boldsymbol{Q}\boldsymbol{R}\boldsymbol{Q}{}^{T},\text{where }\boldsymbol{R}=\ensuremath{\boldsymbol{P}_{\left\lceil \nicefrac{s}{2}\right\rceil }}\ensuremath{\cdots}\ensuremath{\boldsymbol{P}_{1}}\boldsymbol{R}_{*}\ensuremath{\boldsymbol{P}_{1}}^{T}\ensuremath{\cdots}\ensuremath{\boldsymbol{P}_{\left\lceil \nicefrac{s}{2}\right\rceil }^{T}}\text{ and }\text{\ensuremath{\boldsymbol{Q}}=\ensuremath{\boldsymbol{Q}_{*}}\ensuremath{\ensuremath{\boldsymbol{P}_{1}^{T}}}\ensuremath{\cdots}\ensuremath{\boldsymbol{P}_{\left\lceil \nicefrac{s}{2}\right\rceil }^{T}}}.\label{eq:SUD-RSD}
\end{equation}

\section{Optimization for SABRSD\label{sec:SABRSD-code}}

In SABRSD, we minimize $\left\Vert \boldsymbol{I}-\boldsymbol{R}\hat{\boldsymbol{R}}^{-1}\right\Vert $
subjective to $\min\kappa(\boldsymbol{R}\hat{\boldsymbol{R}}^{-1})$.
This optimization is highly nonlinear, and we solved it using the
\textsf{fmincon} function in MATLAB's Global Optimization Toolbox
\cite{MATLAB2020b}. For the $s$-stage GL scheme, $\hat{\boldsymbol{R}}$
has $s\left(s-1\right)/2$ variables for the strictly upper-triangular
part and one variable in the diagonal. Since $\kappa(\boldsymbol{R}\hat{\boldsymbol{R}}^{-1})$
is invariant of scaling of $\hat{\boldsymbol{R}}$, we perform the
optimization in two steps. First, we solved for $\hat{\boldsymbol{R}}_{*}=\arg\min_{\hat{\boldsymbol{R}}}\kappa(\boldsymbol{R}\hat{\boldsymbol{R}}^{-1})$
while fixing the diagonal entries to be $1$. In the second step,
we solve for a scalar factor $\alpha\in(0,1]$ by minimizing $\left\Vert \boldsymbol{I}-\boldsymbol{R}\left(\alpha\hat{\boldsymbol{R}}_{*}\right)^{-1}\right\Vert =\left\Vert \boldsymbol{I}-\frac{1}{\alpha}\boldsymbol{R}\hat{\boldsymbol{R}}_{*}^{-1}\right\Vert $
for $\hat{\boldsymbol{R}}_{*}$ obtained from the first step. For
both steps, we use $10^{-12}$ as the relative convergence tolerances
by default, but a larger threshold (e.g., $10^{-8}$) may also suffice.
 We verified the procedure for the GL schemes with up to 11 stages.

\end{document}